[1994/12/01]
\documentclass{article}
\usepackage{mathrsfs, amsmath,amssymb, mathtools, amsthm}
\usepackage[arrow, matrix, curve]{xy}

\chardef\bslash=`\\ 

\newtheorem{thm}{Theorem}[section]
\newtheorem{cor}[thm]{Corollary}
\newtheorem{lem}[thm]{Lemma}
\newtheorem{prop}[thm]{Proposition}

\theoremstyle{definition}
\newtheorem{defn}{Definition}[section]
\newtheorem{rem}{Remark}[section]
\newtheorem{obs}{Observation}[section]

\theoremstyle{remark}

\newcommand{\eval}[2][\right]{\relax
  \ifx#1\right\relax \left.\fi#2#1\rvert}

\begin{document}
\title{On semisimple $\ell$-modular Bernstein-blocks 
of a $p$-adic general linear group}
\author{David-Alexandre Guiraud\footnote{
Interdisciplinary Center for Scientific Computing,
Heidelberg University, Germany david.guiraud@iwr.uni-heidelberg.de}
}
\maketitle
\vspace{2cm}
\begin{abstract}
Let $G_n=\operatorname{GL}_n(F)$, where $F$ is a non-archimedean local 
field with
residue characteristic $p$. Our starting point is the 
Bernstein-decomposition
of the representation category of $G_n$ over an algebraically closed field
of characteristic $\ell \neq p$ into blocks. 
In level zero, we associate to each block a replacement for the
Iwahori-Hecke algebra which provides a Morita-equivalence just as in 
the complex case. Additionally, we will explain how this gives rise
to a description of 
an arbitrary $G_n$-block in terms of simple $G_m$-blocks (for
$m\leq n$), paralleling the approach of Bushnell and Kutzko in 
the complex setting.
\end{abstract}
\tableofcontents

\section{Introduction}
Consider the group $G = G_n = 
\operatorname{GL}_n(F)$, where $F$ is a local non-archimedean field
with residue characteristic $p$. As part of the Bernstein-decomposition,
it is now a classical result that the category $\mathfrak{R}_{\mathbb{C}}(G)$
of smooth, complex $G$-representations decomposes as
\begin{equation}\label{sdnfoibewogiwbebgfoiwferfajsbviqwe89}
\mathfrak{R}_{\mathbb{C}}(G) = 
\Bigl(\bigoplus_{(\mathscr{P}, \rho)} 
\mathfrak{R}_{\mathbb{C}, (\mathscr{P}, \rho)}
(G)\Bigr)\; \oplus
\text{ positive-level part},
\end{equation}
where
\begin{itemize}
\item $(\mathscr{P}, \rho)$ runs through all equivalence classes
of level-$0$ $G$-types: $\mathscr{P}\subset G$ is
a parahoric subgroup (see Definition \ref{ahfdougbuewgbrfuog398rz327z4237grif})
and $\rho$ a 
$\mathscr{P}$-representation inflated from a supercuspidal 
representation $\overline{\rho}$
of the reductive quotient $\mathscr{P}/\mathscr{P}(1) = \mathcal{M}$. 
Up to equivalence (see Definition \ref{siovnowebogibwoiefhoiwehoifbowie}),
we can assume that $(\mathscr{P}, \rho)$ is in an arranged form, i. e. 
\begin{equation*}
\mathcal{M}= \prod_{i=1}^k
\bigl(\operatorname{GL}_{n_i}(q)\bigr)^{m_i} \text{ and } 
\overline{\rho}= \boxtimes_{i=1}^k (\overline{\rho}_i)^{m_i},
\end{equation*}
where $k, n_i, m_i\in\mathbb{N}$ with $\sum_{i=1}^k n_i m_i = n$ and
each $\overline{\rho}_i$ is a supercuspidal 
$\operatorname{GL}_{n_i}(q)$-representation
with $\overline{\rho}_i\not\cong \overline{\rho}_j$ for $i\neq j$.
\item The subcategory 
$\mathfrak{R}_{\mathbb{C}, (\mathscr{P}, \rho)}(G)$ consists
of representations which have all their irreducible subquotients isomorphic
to subquotients of $\operatorname{ind}_{\mathscr{P}}^G(\rho)$.
\end{itemize}
This reduces the representation theory
of $G$ to an investigation of the blocks 
$\mathfrak{R}_{\mathbb{C}, (\mathscr{P}, \rho)}$. 
Bushnell and Kutzko provided in \cite{bushnell1998smooth}
a Morita-equivalence\footnote{Technically speaking, this is
not a Morita-equivalence as there is just one ring. Anyways, the alternative 
characterisation of (\ref{askvnfoiwebgoiebwiorh3w8i93hbnwiofevbnioho}) as
$\mathscr{H}(G)_{(\mathscr{P}, \rho)} \underset{\text{Morita}}{\cong}
\mathscr{H}(G, \mathscr{P}, \rho)$ 
(with $\mathscr{H}(G)_{(\mathscr{P}, \rho)}$ the 
part of the global Hecke algebra
$\mathscr{H}(G)$ lying in $\mathfrak{R}_{\mathbb{C}, (\mathscr{P}, \rho)}$)
justifies this abuse of notation.}
\begin{equation}\label{askvnfoiwebgoiebwiorh3w8i93hbnwiofevbnioho}
\mathfrak{R}_{\mathbb{C}, (\mathscr{P}, \rho)} \cong
\mathscr{H}(G, \mathscr{P}, \rho)\!-\!\textbf{Mod},
\end{equation}
where $\mathscr{H}(G, \mathscr{P}, \rho) = \operatorname{End}_G\bigl(
\operatorname{ind}_{\mathscr{P}}^G(\rho)\bigr)$ is the Iwahori-Hecke
algebra associated to the type $(\mathscr{P}, \rho)$.
For a
simple type (i. e. with $k=1$), 
they established in \cite{bushnell1993admissible} an isomorphism
\begin{equation}\label{sionweoigbiowebgbweortfbowegbhfouwegbgvof}
\mathscr{H}(G, \mathscr{P}, \rho) \cong 
\mathscr{H}(\operatorname{GL}_{m_1}(F^{n_1}), {I}, 
\mathbb{C}_{\text{triv}}),
\end{equation}
where $F^{n_1}$ is the unramified extension of $F$ with degree $n_1$ and
${I}$ denotes the Iwahori-subgroup of $\operatorname{GL}_{m_1}(F^{n_1})$.
Modules over Hecke algebras of this kind were classified by
Kazhdan-Lusztig and Ginzburg. \\
As a final step, Bushnell and Kutzko decompose in \cite{bushnell1999semisimple}
the Hecke algebra of a general type
as
\begin{equation}\label{asfbnoiwebgfihw0ghoiwegfrbowrbhtguheuigbfiwra}
\mathscr{H}(G, \mathscr{P}, \rho) \cong 
\bigotimes_I \mathscr{H}(G_{m_in_i}, \mathscr{P}_i, \rho_i),
\end{equation}
where all occurring 
pairs $(\mathscr{P}_i, \rho_i)$ are simple $G_{m_in_i}$-types and can henceforth
be treated as in (\ref{sionweoigbiowebgbweortfbowegbhfouwegbgvof}).\vspace{0.3cm}\\
If one replaces the base field $\mathbb{C}$ by some algebraically 
closed field $R$ of positive characteristic $\ell\neq p$, the decomposition
(\ref{sdnfoibewogiwbebgfoiwferfajsbviqwe89}) carries over. 
Although types and Hecke
algebras can still be defined and continue to be an important concept, one
loses the
Morita-equivalence of (\ref{askvnfoiwebgoiebwiorh3w8i93hbnwiofevbnioho}). 
Therefore, a different
concept is needed if one is interested in the structure of the blocks.
The main achievement of this paper is the construction of a 
pair $(\mathscr{P}_{\tt{max}}, \tilde{\rho})$ (called the \textit{supercover}) 
which is a suitable replacement for the type in the sense that it provides
both a Morita-equivalence like (\ref{askvnfoiwebgoiebwiorh3w8i93hbnwiofevbnioho})
and a tensor-decomposition like 
(\ref{asfbnoiwebgfihw0ghoiwegfrbowrbhtguheuigbfiwra}):\\
Let $(\mathscr{P}, \rho)$ be a type given in the arranged form as described
above, then we can form the unique standard parahoric subgroup
$\mathscr{P}_{\tt{max}}$ with reductive quotient 
\begin{equation*}
\mathcal{M}_{\tt{max}} =
\prod_{i=1}^k \operatorname{GL}_{n_im_i}(q).
\end{equation*}
$\mathcal{M}$ is a Levi-subgroup of $\mathcal{M}_{\tt{max}}$ and we
can consider the Harish-Chandra induced 
$\begin{bf}i\end{bf}_{\mathcal{M}}^{\mathcal{M}_{\tt{max}}}(\overline{\rho})$.
Denote by $\Psi$ the set of isomorphism classes of irreducible subquotients
of 
$\begin{bf}i\end{bf}_{\mathcal{M}}^{\mathcal{M}_{\tt{max}}}(\overline{\rho})$.
Any (representative of an) element of $\Psi$ can be written as 
$\boxtimes_{i=1}^k X_i$, where $X_i$ is an irreducible 
representation of $\operatorname{GL}_{n_im_i}(q)$. Any such $X_i$
admits a projective cover $\hat{X}_i$.
Then $\hat{X} \coloneqq \boxtimes_{i=1}^k \hat{X}_i$ can be inflated to a
$\mathscr{P}_{\tt{max}}$-representation
$\tilde{X}$ 
and we set (cf.
Definition
\ref{i3h802hrf8032grt892z4392876432t754t7g})
\begin{equation*}
\tilde{\rho}= \bigoplus_{X\in \Psi} \tilde{X}.
\end{equation*}
In Section \ref{su0938482HDOFH983r09zhf}, we study the 
induced
supercover $\operatorname{ind}^G_{\mathscr{P}_{\tt{max}}}(\tilde{\rho})$.
It follows from certain properties of the Harish-Chandra functor
$\begin{bf}i\end{bf}_{\mathcal{M}}^{\mathcal{M}_{\tt{max}}}$
(collected in Section \ref{eu09w7r34r74zutt6569876iujdsgf}) and a 
special instance of the Mackey-decomposition for parahoric functors
(Section \ref{siodhfioebhgoiweigbewboibrfowe}) that 
$\operatorname{ind}^G_{\mathscr{P}_{\tt{max}}}(\tilde{\rho})$
is a progenerator in 
$\mathfrak{R}_{R, (\mathscr{P}, \rho)}$.
This implies directly the announced Morita-equivalence 
(Theorem \ref{iewoih320h0rth23804u87389435hsdubvuigegt794gsf}):
\begin{equation*}
\mathfrak{R}_{(\mathscr{P}, \rho)}\cong
\mathscr{H}(G, \mathscr{P}_{\tt{max}}, \tilde{\rho})\!-\!\textbf{Mod}
\end{equation*}
The decomposition of $\mathscr{H}(G, \mathscr{P}_{\tt{max}}, \tilde{\rho})$
as a tensor product is technically more involved. The crucial ingredient
is a disjointness argument (Proposition \ref{aishfiowebhgiwbeiogfbiowebfrbweorwebroiweb}) 
based on the work of Dipper, James and Green on (modular) representations
of $\operatorname{GL}_n(q)$.
In Section \ref{io8378492698456328765987sjdvbifwgtr}, we use this 
to first give an upper bound on the intertwining of
\begin{equation*}
\operatorname{Hom}_G\bigl(
\operatorname{ind}^G_{\mathscr{P}_{\tt{max}}}(\tilde{X}),
\operatorname{ind}^G_{\mathscr{P}_{\tt{max}}}(\tilde{Y})\bigr)
\end{equation*}
for arbitrary $X, Y\in \Psi$.
Using the general methods of 
Section \ref{aohfoweiht08zr84z3r849tgr94g0}, we are able to lift this
to an upper bound on the intertwining of 
$\mathscr{H}(G, \mathscr{P}_{\tt{max}}, \tilde{\rho})$ 
(see Theorem \ref{sdoihfoiwhbgfiwebhfgbwei}). This is sufficient
to use an argument of Vign{\'e}ras implying the tensor-decomposition:
Denote by $(\mathscr{K}_i, \tilde{\rho}_i)$ the supercover of
 the simple $\operatorname{GL}_{n_im_i}(F)$-type 
$(\mathscr{P}_i, \rho_i)$, where
\begin{itemize}
\item $\mathscr{K}_i = \operatorname{GL}_{n_im_i}(\mathcal{O})$;
\item $\mathscr{P}_i$ is the unique parahoric subgroup of 
$\operatorname{GL}_{n_im_i}(F)$ with reductive quotient
$\mathcal{M}_i = \bigl(\operatorname{GL}_{n_i}(q)\bigr)^{m_i}$;
\item $\rho_i$ is inflated from the $\mathcal{M}_i$-representation
$\overline{\rho}_i^{m_i}$.
\end{itemize} 
Then we establish
(cf. Theorem \ref{saoihfoih89thg2984gt98234rgh92trg5942rvgb9uvb9q87gwtgb4})
\begin{equation*}
\mathscr{H}(G, \mathscr{P}_{\tt{max}},\tilde{\rho}) 
\underset{\text{Morita}}{\cong}
\bigotimes_I 
\mathscr{H}(\operatorname{GL}_{n_im_i}(F), \mathscr{K}_i, \tilde{\rho}_i).
\end{equation*}
We repeat that
\begin{eqnarray*}
\mathfrak{R}_{(\mathscr{P},\rho)}(G) \cong
\mathscr{H}(G, \mathscr{P}_{\tt{max}},\tilde{\rho})\!-\!\textbf{Mod},
\qquad\qquad\\
\mathfrak{R}_{(\mathscr{P}_i,\rho_i)}\bigr(
\operatorname{GL}_{n_im_i}(F)\bigl) \cong
\mathscr{H}(\operatorname{GL}_{n_im_i}(F), \mathscr{K}_i, \tilde{\rho}_i)\!-\!\textbf{Mod}.
\end{eqnarray*}
This reduces the study of a general 
(called \textit{semisimple}) block
to simple blocks just as Bushnell-Kutzko theory does
in the complex setting. Expressed in sloppy words, 
this tells us that semisimple blocks
are built up from simple ones in the easiest possible way, i. e. all 
`{mysterious}'
things happen in the formation of simple blocks from their
supercuspidal parts. This is reflected by the
fact that cuspidal non-supercuspidal representations (whose existence
is a unique feature of the modular case) can occur only in  
simple blocks\footnote{This follows from a (possibly iterated)
application of \cite{VigLivre}, Claim III.3.15.2.}. Moreover, from 
the definition of the supercover it is clear that all modular complications
in the representation theory of $G$ come from modular complications
in the representation theory of finite linear groups: If for two choices
$\ell, \ell'$ the representation theories of 
$\operatorname{GL}_m(q)$ are identical (and we ask for this to hold 
for all $m\leq n$), 
then the level-$0$ parts of the representation theories of $G$ are identical over $\ell$
and over $\ell'$.\vspace{0.3cm}\\
In Section \ref{ewihfih08ghtf832gbtfug32b984g732g4rq09t},
we demonstrate how our technique can be used to study the smallest non-trivial 
example: 
\begin{itemize}
\item $G= \operatorname{GL}_2(\mathbb{Q}_p)$;
\item $(I, \rho)$ with $I$ the Iwahori subgroup
and $\rho$ inflated from 
a $\mathbb{F}_p^{\times}\times \mathbb{F}_p^{\times}$-character
$\overline{\rho}_1\otimes
\overline{\rho}_2$ with $\overline{\rho}_1 \not\cong
\overline{\rho}_2$.
\end{itemize}
There is a subgroup 
$(\mathbb{Z}_p^{\times})^{(\ell)}$
of $\mathbb{Z}_p^{\times}$ with pro-order prime to $\ell$ such
that the quotient is a finite $\ell$-group. We can define
the subgroup 
\begin{equation*}
I^{(\ell)} = \begin{pmatrix}
(\mathbb{Z}_p^{\times})^{(\ell)} & \mathbb{Z}_p \\
\mathfrak{P} & (\mathbb{Z}_p^{\times})^{(\ell)}
\end{pmatrix}\subset I
\end{equation*}
and prove
\begin{equation*}
\mathfrak{R}_{(I, \rho)}(G) \cong
\mathscr{H}(G, I^{(\ell)}, \rho|I^{(\ell)})\!-\!\textbf{Mod}.
\end{equation*}
Additionally, we get
\begin{equation*}
\mathscr{H}(G, I^{(\ell)}, \rho|I^{(\ell)}) 
\cong \bigotimes_{\text{two copies}} R\left[\mathbb{Q}_p/
(\mathbb{Z}_p^{\times})^{(\ell)}\right],
\end{equation*}
where the category of modules over 
$R\left[\mathbb{Q}_p/
(\mathbb{Z}_p^{\times})^{(\ell)}\right]$ is equivalent to the unipotent
block (i. e. the block containg the trivial representation) of
$\mathfrak{R}\bigl(\operatorname{GL}_1(F)\bigr)$.\vspace{0.3cm}\\
It is desireable to generalise results about $\operatorname{GL}_n(F)$
to arbitrary reductive $p$-adic groups. In our case, the first
obstacle would be the Bernstein-decomposition which gets more complicated
(cf. Thm. III.6 in \cite{vignéras1998induced}). Although it might still
be possible
to define the supercover in some more general situations, it is ultimately
the use of certain results on $\operatorname{GL}_n(q)$ (which are not
available in greater generality) what limits our techniques
to the situation studied in this paper.
This doesn't come as a surprise, given that in the complex case
the Bushnell-Kutzko results can be generalised 
neither easily nor completely to reductive groups. 
Another question is whether there are 
connections between the supercover of a simple type and the
supercover of $(I, R_{\text{triv}})$ for some other group, paralleling
(\ref{sionweoigbiowebgbweortfbowegbhfouwegbgvof}) in the complex case.
Together with the task of generalising the presented results
to positive-level blocks, this poses 
interesting topics for future research.\vspace{0.3cm}\\
This work was partly inspired
by conversations with Marie-France Vign{\'e}ras. 
The major part of
the research was conducted during a visiting
stay at Bar-Ilan University.
The author wants to thank Michael Schein for hosting this stay, for his
valuable support and for his comments on this paper.
The author wants to thank Maarten Solleveld
for comments and corrections on this paper,
leading (among other things) to the present formulation of Proposition
\ref{snvoiwboivgbeifbgeiwbgvbdwoigvboweobiiojsdbvajb}.v.

\section{Preliminaries}\label{weivoiwebhgiohwrohguagberfbwisdaifvpbwgf}
Fix two prime numbers $p\neq \ell$. Let $F$ be a local 
non-archimedean field with 
ring of integers $\mathcal{O}$, some fixed uniformiser $\varpi$, 
maximal ideal $\mathfrak{P}= \varpi\mathcal{O}$ and 
residue field $k \cong \mathbb{F}_q$, where $q$ is some power
of $p$.  
Moreover, let $R$ be an algebraically closed field of characteristic
$\ell$ such that $R$ arises as residue field in some \mbox{$\ell$-modular}
system $(R, \mathcal{O}_K, K)$.\vspace{0.15cm}\\
The group $G_n = \operatorname{GL}_n(F)$ inherits a topology from
$F$ 
and provides an example of what
one calls a locally profinite group (cf. \cite{VigLivre}). 
\begin{defn}[Smooth representation]
An $R$-valued representation $(V, \pi)$ of a locally 
profinite group $G$ is called smooth
if we can find for any $v\in V$ an open subgroup $K\subset G$ which
acts trivially on $v$. Together with $G$-equivariant
linear maps, smooth $G$-representations
define a category denoted by $\mathfrak{R}(G)$.
\end{defn}
Depending on the context, we will refer to a representation $(V, \pi)$
by simply writing $V$ or $\pi$. From now on, we will only be concerned
with $R$-valued smooth representations. It is a basic observation (cf.
\cite{VigLivre}) 
that $\mathfrak{R}(G)$ is equivalent to the category of modules 
over the global Hecke algebra
\begin{equation*}
\mathscr{H}(G) = \{f: G\rightarrow R \,|\, f \text{ locally constant
and compactly supported}\}
\end{equation*}
where multiplication is defined as convolution with respect
to some chosen Haar measure on $G$ (see \cite{VigLivre}, I.3.1). This parallels
the interpretation of representations of a finite group as modules over
the group algebra.
\begin{defn}[Induction with compact support]
Let $H\subset G$ be a closed subgroup of a locally profinite group 
and consider an $H$-representation $(V,\pi)$. Define 
$\operatorname{ind}_H^G(\pi)_{\ast}$ as the space
\begin{equation*}
\{f:G\rightarrow V \,|\, f(hg) = \pi(h)f(g)
\,\forall h\in H, g\in G, f \text{ compactly supported mod-$H$}\}.
\end{equation*}
The $G$-action $f\overset{g}{\longmapsto} f(\textvisiblespace \,g)$
allows us to consider $\operatorname{ind}_H^G(\pi)_{\ast}$ as a (possibly 
not smooth) $G$-representation. Then define
$\operatorname{ind}_H^G(\pi)$ as the biggest subrepresentation
of $\operatorname{ind}_H^G(\pi)_{\ast}$ which is smooth. 
(This last step is trivial if $H$ happens to be open.)
\end{defn}
There is an obvious way how $\operatorname{ind}_H^G$ acts on arrows,
allowing us to view induction with compact support as a functor
$\mathfrak{R}(H)\rightarrow \mathfrak{R}(G)$. 
If $H$ is open, this functor corresponds (by \cite{VigLivre}, I.5.2) to
\begin{equation}\label{aosfoibwefgbghb3htf08iwh48itfh0328htgb02}
\mathscr{H}(H)\!-\!\textbf{Mod} \longrightarrow 
\mathscr{H}(G)\!-\!\textbf{Mod}\qquad
M \mapsto \mathscr{H}(G)\otimes_{\mathscr{H}(H)}M.
\end{equation}
Let us collect the basic properties:
\begin{prop} \label{aioshoihbfohgbweoifgboiubosdbfoasbf}
Let $H$ be a closed subgroup of $G$ and denote the
restriction of a $G$-representation to an $H$-representation by
$\operatorname{res}_H^G$. Then
\renewcommand{\labelenumi}{\roman{enumi})}\label{aisfhoiwehbgoibweiogbh}
\begin{enumerate}
\item Both $\operatorname{res}_H^G$ and $\operatorname{ind}_H^G$
commute with direct sums;
\item We have the following adjointness properties: Let $V\in \mathfrak{R}(G)$
and $W\in \mathfrak{R}(H)$, then
\begin{equation*}
\operatorname{Hom}_G\bigl(\operatorname{ind}_H^G(W), V\bigr) \cong
\operatorname{Hom}_H\bigl(W, \operatorname{res}_H^G(V)\bigr)
\text{ if $H$ is open in $G$}
\end{equation*}
and
\begin{equation*}
\operatorname{Hom}_G\bigl(V,\operatorname{ind}_H^G(W)\bigr) \cong
\operatorname{Hom}_H\bigl(\operatorname{res}_H^G(V),W\bigr)
\text{ if $H$ is co-compact in $G$;}
\end{equation*}
\item If $H$ is open in $G$, the functor $\operatorname{ind}_H^G$ 
respects the properties 
`{cyclic}', `finitely generated' and `{projective}';
\item Let both $H_1, H_2$ be open subgroups, then we have a Mackey-decomposition
\begin{equation*}
\operatorname{res}_{H_2}^G\circ \operatorname{ind}_{H_1}^G
\cong \bigoplus_{g\in H_1\!\backslash G/H_2} \operatorname{ind}_{H_2\cap
H_1^{g^{-1}}}^{H_2} \circ \operatorname{Int}(g)\circ
\operatorname{res}_{H_1\cap
H_2^{g}}^{H_1};
\end{equation*}
\item The functor $\operatorname{res}_H^G$ is exact; 
Assume there exists an open compact subgroup $K^{\ast}\subset G$ whose
pro-order is different from zero in $R$. Then $\operatorname{ind}_H^G$
is exact as well.
\end{enumerate}
\end{prop}
\begin{proof}
The only claim not literally taken from Chapter I.5 of \cite{VigLivre} is the 
`finitely generated'-part of iii), but this is clear from the characterisation
(\ref{aosfoibwefgbghb3htf08iwh48itfh0328htgb02}).
\end{proof} 
\begin{rem}
In all cases we are interested in (i. e. $G$ a linear algebraic
$p$-adic group and $\ell\neq p$), the assumption of part v) is fulfilled,
see \cite{meyer2009resolutions}, Lemma 1.1.
\end{rem}

\subsection{Representations of direct products}
In this section we will deal with the connections between representations
of $G$, $G'$ and $G\times G'$, where $G, G'$ are two groups. The basic tool
will be
\begin{defn}[Outer tensor product]
Let $V$ be a representation of $G$ and $V'$ be a representation of $G'$.
The group $G\times G'$ acts on the space $V\otimes_R V'$ by linear
continuation of the rule
\begin{equation*}
(g,g')\,v\otimes v' \coloneqq gv\otimes g'v'.
\end{equation*}
The resulting representation of $G\times G'$ is called the outer
tensor product and referred to by the symbol $V\boxtimes V'$.
\end{defn}
In the two following propositions, by `linear group' we mean a direct
product of finitely many general linear groups. 
\begin{prop}\label{eohiwohgieoh3bi0hr048ihri24tih08rtfhgurfpobu94g}
The following gives a description of the outer tensor product which 
generalises to finitely many groups and tensor factors:
\renewcommand{\labelenumi}{\roman{enumi})}
\begin{enumerate}
\item Let $G, G'$ be finite groups, then there is an isomorphism of algebras
\begin{equation*}
R[G]\otimes_R R[G'] \cong R[G\times G']
\end{equation*}
and under this characterisation $V\boxtimes V'$ corresponds to the
space $V\otimes V'$ on which $R[G]\otimes_R R[G']$ acts by linear
continuation of the rule
\begin{equation*}
(f\otimes f')\cdot(v\otimes v') \coloneqq (f\cdot v) \otimes (f'\cdot v').
\end{equation*}
\item Let $G, G'$ be two linear $p$-adic groups and $V$ a (smooth)
$G$-representation and $V'$ a (smooth) $G'$-representation. Then 
$V\boxtimes V'$ is smooth and
there
is an isomorphism of algebras
\begin{equation*}
\mathscr{H}(G) \otimes_R \mathscr{H}(G') \cong \mathscr{H}(G\times G').
\end{equation*}
Under this characterisation, $V\boxtimes V'$ corresponds to the
space $V\otimes V'$ on which 
$\mathscr{H}(G) \otimes_R \mathscr{H}(G')$ acts by linear
continuation of the rule
\begin{equation*}
(\varphi \otimes \psi) \ast (v\otimes v') \coloneqq 
(\varphi \ast v) \otimes (\psi\ast v').
\end{equation*}
\end{enumerate}

\end{prop}
\begin{proof}For part i) we refer to Chapter 2.6 of
\cite{karpilovsky1990induced}. \\
Considering part ii), we first mention
that it is straight-forward to check that $V\boxtimes V'$ is smooth.
The isomorphism of the Hecke algebras is provided by linear
continuation of the rule
\begin{equation*}
\varphi\otimes \psi \mapsto \bigl[ (g, g') \mapsto \varphi(g)\cdot \psi(g')
\bigr].
\end{equation*}
This assignment provides a well-defined and bijective map, 
as shown in \cite{guiraud}, Prop 4.5.5. It follows from a standard
Fubini-style theorem (see \cite{guiraud}, Thm. 1.4.14 and Remark 2.5.11) that
this map commutes with the $\ast$-multiplication. 
The remaining claim is checked readily using the definitions.
\end{proof}
We repeat that, when dealing with $p$-adic groups, we will assume that all
representations under consideration are smooth without mentioning this 
each and every time. Let us collect some formal properties:
\begin{prop}\label{snvoiwboivgbeifbgeiwbgvbdwoigvboweobiiojsdbvajb}
Consider $G$-representations $V, V_i, W$ and $G'$-representations 
$V', W'$, where $G$ and $G'$ are either both finite or both $p$-adic
linear groups. Then:
\renewcommand{\labelenumi}{\roman{enumi})}
\begin{enumerate}
\item The outer tensor product is distributive:
\begin{equation*}
(V\oplus W) \boxtimes V' = V\boxtimes V' \oplus W\boxtimes V'
\end{equation*} 
and analogously in the second variable;
\item $V\boxtimes V'$ is cyclic if both $V$ and $V'$ are cyclic (and
this is also true if we replace `{cyclic}' by 
`{finitely generated}' or by `finitely generated and projective');
\item If we have a sequence of $G$-modules
\begin{equation*}
\cdots \rightarrow V_{j-1} \rightarrow V_j \rightarrow V_{j+1} 
\rightarrow \cdots
\end{equation*}
which is exact at $j$, then so is the sequence
\begin{equation*}
\cdots \rightarrow V_{j-1}\boxtimes V'  
\rightarrow V_j\boxtimes V' \rightarrow V_{j+1}\boxtimes V'
\rightarrow \cdots
\end{equation*}
at $j$ for any $V'$ (as a sequence of $G\times G'$-modules).
The same is true if we fix the $G$-factor and take
a sequence of $G'$-representations.
\item 
Assume that $V$ and $V'$ are finite-length, then 
define $\Gamma$ to be the set of all decomposition factors (up to
isomorphism) of $V$ and $\Gamma'$ analogously for $V'$, 
then $V\boxtimes V'$ is finite-length and every subquotient
is isomorphic to $Q\boxtimes Q'$ for suitable 
$Q\in \Gamma, Q'\in \Gamma'$; 
\item Let $f:V \rightarrow V'$ and 
$f':W\rightarrow W'$ be surjective maps, then so is
the induced map $f\boxtimes f': V \boxtimes W \rightarrow V'\boxtimes W'$;
\item Formation of the outer tensor product provides a bijection
\begin{equation*}
\operatorname{Irr}_R(G) \times \operatorname{Irr}_R(G')
\;\overset{1:1}{\longleftrightarrow}\; \operatorname{Irr}_R(G\times G').
\end{equation*}
\end{enumerate}
The obvious analogues of these assertions hold if we take $G, G', G'', \ldots$
some finite collection of groups.
\end{prop}
\begin{proof}
Regarding i), the map given by linear continuation of
\begin{equation*}
f: (v,w)\otimes v' \mapsto (v\otimes v', w\otimes v')
\end{equation*}
is known to provide an isomorphism of vector spaces. 
$f$ is obviously $G\times G'$-equivariant.\\
From now on, write $R_G$ for $R[G]$ (if $G$ is finite) or 
$\mathscr{H}(G)$ (if $G$ is $p$-adic). The 
`{cyclic}'- and `{finitely generated}'-claims of ii) will follow
immediately from v). So let both $V$ and $V'$ be
finitely generated and projective. Then, by 
Chapter II, Paragraph 2, no. 2, Cor. to Prop. 4 of \cite{Bourbaki}, 
$V\oplus Q = R_G^m$ and
$V'\oplus Q' = R_{G'}^{m'}$, for a $G$-representation $Q$, 
a $G'$-representation $Q'$ and two numbers $m, m'\in \mathbb{N}$.
Then 
\begin{equation*}
(V\boxtimes V') \oplus \bigl( (V\boxtimes Q') \oplus (V'\boxtimes Q) \oplus
(V'\boxtimes Q')\bigr) = R_G^m \otimes R_{G'}^{m'} = R_{G\times G'}^{m
m'},
\end{equation*}
hence $V\boxtimes V'$ is finitely generated and projective.\\
The proof for iii) is analogous to i): This claim is true 
for $R$-vector spaces and the occurring maps are easily seen to be 
$G\times G'$-equivariant. iv) is a direct consequence of iii).
For v), it is straight-forward to construct a pre-image for
any element in $V'\otimes W'$. vi) 
follows by putting together Proposition 
\ref{eohiwohgieoh3bi0hr048ihri24tih08rtfhgurfpobu94g} and
\cite{bump1998automorphic}, Theorem 3.4.2 (and \cite{VigLivre}, II.2.8, in
the $p$-adic case).
\end{proof}
\begin{lem}\label{eiheg509hw8hfw2h4gf8h2498eghf84e29ghf8924htf}
In the notation of the preceding proposition, assume that both
$V$ and $V'$ are finitely generated and projective. Then
there is a ring-homomorphism
\begin{equation*}
\operatorname{End}_G(V)\otimes_R \operatorname{End}_{G'}(V') \cong
\operatorname{End}_{G\times G'}(V\boxtimes V').
\end{equation*}
This generalises to the case where we consider a finite collection
of groups $G, G', G'', \ldots$.
\end{lem}
\begin{proof}
Using the characterisation of representations as modules
over the appropriate group (or Hecke-) algebra, we can conclude the 
claim from Exercise 5 in Chapter 9.3 of \cite{Pierce}
when both $V$ and $V'$ are finitely presented. But this is the case,
as finitely generated projective implies
finitely presented (the proof of \cite{BourbakiCom}, 
Chapter 1, Paragraph 2, no. 8, Lemma 8.iii works for non-commutative
rings). Part ii of Proposition 
\ref{snvoiwboivgbeifbgeiwbgvbdwoigvboweobiiojsdbvajb} provides an 
iterative way of generalising the claim to arbitrary finite collections.
\end{proof}
From now on, we restrict ourselves to the $p$-adic setting:
\begin{prop}\label{wihfgoiwbgwighni4ehbio4hbgoihno2i4hng42}
Consider two open, compact subgroups 
$K\subset G$, $K'\subset G'$ and let $V$ be a $K$-representation and
$V'$ be a $K'$-representation. Then we have
\begin{equation}\label{iodshohg2380rf2zr93ztgufwe9g283g98r34g9g239g}
\operatorname{ind}_K^G(V)\boxtimes 
\operatorname{ind}_{K'}^{G'}(V') \cong
\operatorname{ind}_{K\times K'}^{G\times G'} (V\boxtimes V').
\end{equation}
\end{prop}
\begin{proof}
We define a map $\alpha$ from the 
left to the 
right side by linear continuation of
the following rule:
\begin{equation*}
f\boxtimes f' \longmapsto \Big[ (g,g')\mapsto f(g)\boxtimes f'(g')\Big]
\end{equation*}
It is obvious that this is a well-defined $G\times G'$-intertwiner.\\
The right side of (\ref{iodshohg2380rf2zr93ztgufwe9g283g98r34g9g239g})
is linearly spanned by functions supported on a single coset
$(Kg, K'g')$. Let $\varphi$ be such a function, then it is characterized
by its value at $(g,g')$, say, $\varphi(g,g')= \sum_i v_i\boxtimes v_i'$.
Now we can define a map from the right side to the left side
of (\ref{iodshohg2380rf2zr93ztgufwe9g283g98r34g9g239g}) by linear continuation
of the rule sending $\varphi$ to $\sum_i f_i\boxtimes f_i'$ with
all the $f_i$ supported on $Kg$ and the $f_i'$ supported on $K'g'$ and
$f_i(g)=v_i, f_i'(g')=v_i'$. It is easy to see that these maps are inverse
to each other.\end{proof}
\begin{rem}
If both $G$ and $G'$ are general linear groups over $F$, 
 the categories $\mathfrak{R}(G)$ and 
$\mathfrak{R}(G')$ are noetherian by
Th\'eor\`eme 5.4.1 in \cite{dat1999types}.
Thus, in this case, the claim 
of Lemma \ref{eiheg509hw8hfw2h4gf8h2498eghf84e29ghf8924htf} holds
if $V$ and $V'$ are just finitely generated. It might be possible to generalise
this to finitely many groups and representations.
\end{rem}

\subsection{Parabolic and parahoric functors}
Recall that, in general for the group $\operatorname{GL}_n(K)$ over some
field $K$, a parabolic subgroup is defined to be the stabiliser of a 
flag in the vector space $K^n$. If this flag is adapted to the standard
basis we call the resulting parabolic subgroup \textit{standard}. 
A partition $\lambda$ of $n$ gives rise to a standard parabolic subgroup,
for example associated to $\lambda= (1, 1, \ldots, 1)$ we have the
standard Borel subgroup of upper-triangular matrices. Any parabolic 
subgroup is $\operatorname{GL}_n(K)$-conjugate 
to a standard one. Recall moreover, that
each parabolic subgroup $P$ decomposes as $P= M \ltimes U$, where $U$
is the unipotent radical of $P$ and $M \cong \prod_I \operatorname{GL}_{n_i}(K)$
(for suitable numbers $n_i$ with $\sum_I n_i =n$) is the Levi-factor of $P$.
\begin{defn}[Parabolic induction and restriction]
Let $P=MU \subset G = \operatorname{GL}_n(K)$ be a parabolic subgroup, where
$K=F$ or $K=\mathbb{F}_q$. The parabolic induction functor then transforms
an $M$-representation into a $G$-representation by first inflating
trivially along $U$ (what yields a $P$-representation) and then inducing up to 
$\operatorname{GL}_n(K)$ 
(induction with compact support in the $p$-adic case).\\
The parabolic restriction associates to a 
$\operatorname{GL}_n(K)$-representation $V$ the space $V_U = V/V(U)$, where
\begin{equation*}
V(U) = \langle v-uv\,|\, v\in V, u\in U\rangle
\end{equation*}
is the space of $U$-coinvariants. $V_U$ is naturally a representation of $M$.\\
In the $p$-adic case, these functors are called Jacquet functors and we write
$\textit{i}_{M\subset P}^G, \textit{i}_{M\subset P}^G$.
In the finite case, these functors are called Harish-Chandra functors and
we write $\begin{bf}i\end{bf}_{{M}\subset P}^{{G}},
\begin{bf}r\end{bf}_{{M}\subset P}^{{G}}$ or simply $\begin{bf}i\end{bf}_{{M}}^{{G}},
\begin{bf}r\end{bf}_{{M}}^{{G}}$. This last notation
is justified by the Howlett-Lehrer theorem (cf. \cite{howlett1994harish}), which
asserts that the isomorphism class of the induced (or restricted) 
representation depends only on the Levi-factor and not on the particular
parabolic subgroup (which is not true in the $p$-adic setting).
\end{defn}
\begin{defn}[(Super-)cuspidal representation]
An irreducible representation $V$ of $G$ is called cuspidal if 
$\textit{r}_{M\subset P}^G(V)\neq 0$ (resp. 
$\begin{bf}r\end{bf}_{{M}}^{{G}}(V)\neq 0$) implies $M=G$.
$V$ is called supercuspidal if its occurrence as a subquotient
of some
$\textit{r}_{M\subset P}^G(W)$ (resp. $\begin{bf}i\end{bf}_{{M}}^{{G}}(W)$)
implies $M=G$.
\end{defn}

We collect the basic facts:
\begin{prop} \label{wnogivbweiogbiowebghtnwehgfoibweoigbvoibihnoi}
Let $G=\operatorname{GL}_n(K)$ be as above.
\renewcommand{\labelenumi}{\roman{enumi})}
\begin{enumerate}
\item Let $P =MU$ be a parabolic subgroup of $G$ and
$Q =NV$ be a parabolic subgroup of $M$. Then $NVU$ is a parabolic
subgroup of $G$ and we have
\begin{equation*}
\textit{i}_{N\subset NVU}^G \cong 
\textit{i}_{M\subset P}^G\circ
\textit{i}_{N\subset Q}^M\quad
(\text{ resp. }\, 
\begin{bf}i\end{bf}_{{N}}^{{G}} \cong
\begin{bf}i\end{bf}_{{M}}^{{G}}\circ \begin{bf}i\end{bf}_{{N}}^{{M}})
\end{equation*}
and 
\begin{equation*}
\textit{r}_{N\subset NVU}^G \cong 
\textit{r}_{N\subset Q}^M\circ
\textit{r}_{M\subset P}^G\quad
(\text{ resp. }\, 
\begin{bf}r\end{bf}_{{N}}^{{G}} \cong
\begin{bf}r\end{bf}_{{N}}^{{M}}\circ \begin{bf}r\end{bf}_{{M}}^{{G}});
\end{equation*}
\item The Harish-Chandra functors commute with taking the contragredient
representation;
\item Let $P=MU$ and $Q=NV$ be two standard parabolic subgroups of $G$.
Then $M\cap Q$ is a parabolic subgroup of $M$ with unipotent radical
$M\cap U$. We have the following Mackey-style decomposition:
\begin{equation*}
\begin{bf}r\end{bf}_{{N}}^{{G}}\circ \begin{bf}i\end{bf}_{{M}}^{{G}}
= \bigoplus_{g\in P\backslash G/Q} \begin{bf}i
\end{bf}_{{N\cap P^{g^{-1}}}}^{{N}}
\circ \operatorname{Int}(g)\circ \begin{bf}r\end{bf}^{{M}}_{{M\cap Q^g}}.
\end{equation*}
\item Let $V$ be a $G$-representation and $W$ be an $M$-representation,
then we have the following adjointness relations:
\begin{eqnarray*}
\operatorname{Hom}_M(\textit{r}_{M\subset P}^G(V), W) \cong
\operatorname{Hom}_G(V, \textit{i}_{M\subset P}^G(W)),\\ 
\operatorname{Hom}_M(\begin{bf}r\end{bf}_{M}^G(V), W) \cong
\operatorname{Hom}_G(V, \begin{bf}i\end{bf}_{M}^G(W));
\end{eqnarray*}
In the finite case, we have moreover
\begin{equation*}
\operatorname{Hom}_G(\begin{bf}i\end{bf}_{M}^G(W),V )\cong
\operatorname{Hom}_M(W, \begin{bf}r\end{bf}_{M}^G(V)).
\end{equation*}
\end{enumerate}
\end{prop}
\begin{proof}
Everything can be extracted from \cite{VigLivre} or \cite{BDC}.
\end{proof}

In the sequel, the letters $\mathcal{G}, \mathcal{P}, \mathcal{M},\ldots$
will always denote finite groups whereas the notation $G, P, M, \ldots$ 
will be reserved for the $p$-adic case. For example, we will use the
short-hand notation $G_n$ for $\operatorname{GL}_n(F)$ and 
$\mathcal{G}_n$ for $\operatorname{GL}_n(q)$.\vspace{0.1cm}\\
Consider once again the maximal compact subgroup 
\begin{equation*}
\mathscr{K} = \operatorname{GL}_n(\mathcal{O}) \subset G= G_n.
\end{equation*}
If we denote by $\mathscr{K}(1) = \begin{textbf}1\end{textbf} + \mathbb{M}_{n\times n}
(\mathfrak{P})$ the pro-$p$-radical, reduction modulo $\mathscr{K}(1)$ gives a
group-homomorphism
\begin{equation*}
f_{\mathscr{K}}: \mathscr{K} \twoheadrightarrow \mathcal{G} = \mathcal{G}_n.
\end{equation*}
\begin{defn}[Parahoric subgroups]\label{ahfdougbuewgbrfuog398rz327z4237grif}
Let $\mathcal{P} = \mathcal{M}\mathcal{U}$ be a standard parabolic subgroup of 
$\mathcal{G}$. Then the pre-image 
$\mathscr{P}= f_{\mathscr{K}}^{-1}(\mathcal{P})\subset G$ is called
a standard parahoric subgroup. $\mathscr{P}$ is open and compact
with pro-$p$-radical $\mathscr{P}(1) = f_{\mathscr{K}}^{-1}(\mathcal{U})$.
Taking the quotient modulo $\mathscr{P}(1)$ defines a group-homomorphism
\begin{equation*}
f_{\mathscr{P}}: \mathscr{P} \twoheadrightarrow \mathcal{M}.
\end{equation*}
A parahoric subgroup is a group $G$-conjugate to a standard parahoric
subgroup.
\end{defn}
\begin{defn}[Parahoric induction and restriction]
Retain the notation from the above definition. Then the parahoric
induction functor maps an $\mathcal{M}$-rep\-re\-sen\-ta\-tion to a $G$-representation
by firstly inflating trivially along $\mathscr{P}(1)$ (what yields
a $\mathscr{P}$-representation) and then inducing compactly to $G$. The
parahoric restriction functor sends a $G$-representation to 
its space of $\mathscr{P}(1)$-invariants, which is then understood as 
a representation of $\mathscr{P}/\mathscr{P}(1) \cong \mathcal{M}$.
We denote these functors by $\textit{i}_{\mathscr{P}}^G, 
\textit{r}_{\mathscr{P}}^G$.
If we are only interested in standard parahoric subgroups, we will
use the notation $\textit{i}_{\mathcal{M}}^G, 
\textit{r}_{\mathcal{M}}^G$, where $\mathcal{M}$ is a standard 
Levi-subgroup of $\mathcal{G}$
(the dropping of $\mathscr{P}$ in the
notation will be justified by Proposition
\ref{sofhiowbeviobweoifbwoeihfewhfowehoigf}.ii).
\end{defn}
The standard facts are as follows:
\begin{prop}\label{sofhiowbeviobweoifbwoeihfewhfowehoigf}
Retain the notation from the above definitions, then
\renewcommand{\labelenumi}{\roman{enumi})}
\renewcommand{\labelenumii}{\arabic{enumii}.}
\begin{enumerate}
\item Both functors are exact and  
parahoric induction commutes with finite direct sums and
respects the properties `{cyclic}', `{finitely generated}' and
`{projective}';
\item Let $\mathcal{N}\subset \mathcal{M}\subset \mathcal{G}$, where each
inclusion is an inclusion of Levi-subgroups. Then 
\begin{equation*}
\textit{i}_{\mathcal{N}}^G \cong \textit{i}_{\mathcal{M}}^G\circ
\begin{bf}i\end{bf}_{\mathcal{N}}^{\mathcal{M}} \text{ and }
\textit{r}_{\mathcal{N}}^G \cong 
\begin{bf}r\end{bf}_{\mathcal{N}}^{\mathcal{M}}
\circ\textit{r}_{\mathcal{M}}^G;
\end{equation*}
\item Let $\mathcal{M}, \mathcal{N}$ be two Levi-subgroups of
$\mathcal{G}$, associated to parahoric subgroups $\mathscr{P}, \mathscr{Q}$ 
of $G$. Then there is a Mackey-decomposition
\begin{equation*}
\textit{r}_{\mathcal{N}}^G\circ
\textit{i}_{\mathcal{M}}^G  \cong 
\bigoplus_{g\in \mathscr{Q}\backslash G/\mathscr{P}}
F^{G}_{\mathscr{Q}(q)g\mathscr{P}(q)},
\end{equation*}
where $F^{G}_{\mathscr{Q}(q)g\mathscr{P}(q)}$ is the functor given by
concatenation of
\begin{enumerate}
\item Parabolic restriction along $f_{\mathscr{P}}
(\mathscr{P}\cap g^{-1}\mathscr{Q}g)\subset \mathcal{M}$, i. e. taking the 
$f_{\mathscr{P}}
(\mathscr{P}\cap g^{-1}\mathscr{Q}(1)g)$-(co)invariants;
\item The functor between representations of the reductive quotients
induced by conjugation
\begin{equation*}
\operatorname{Int}(g): \mathscr{P}\cap g^{-1}\mathscr{Q}g\rightarrow
g\mathscr{P}g^{-1} \cap
\mathscr{Q}
\end{equation*} (cf. \cite{vigneras2003schur}, 4.1.1.(e)));
\item Parabolic induction along $f_{\mathscr{Q}}(g\mathscr{P}g^{-1} \cap
\mathscr{Q})$;
\end{enumerate}
\item The set $D = S_n\cdot \Lambda \subset G$ with 
\begin{equation*}
\Lambda = \bigl\{\operatorname{diag}(\varpi^{a_1}, \ldots, \varpi^{a_n})
\;\bigl|\; a_1, \ldots, a_n\in \mathbb{Z}\bigr\} \cong \mathbb{Z}^n
\end{equation*}
is a set of representatives for $\mathscr{I}\backslash G/\mathscr{I}$
(with $\mathscr{I}= f_{\mathscr{K}}^{-1}(\text{Borel})$ the Iwahori-subgroup).
As $\mathscr{I}$ is contained in any standard parahoric subgroup, we can
(e. g. for the purpose of the Mackey-decomposition) 
choose a set of representatives for $\mathscr{Q}\backslash G/\mathscr{P}$
inside $D$;
\item Let $V$ be a $G$-representation and $W$ be an 
$\mathcal{M}$-representation,
then we have an adjointness relation
\begin{equation*}
\operatorname{Hom}_G(\textit{i}_{\mathcal{M}}^G(W), V) \cong
\operatorname{Hom}_{\mathcal{M}}(W, \textit{r}_{\mathcal{M}}^G(V)).
\end{equation*}
\end{enumerate}
\end{prop}
\begin{proof}
The claim about exactness of induction follows from
Proposition \ref{aioshoihbfohgbweoifgboiubosdbfoasbf}.v
because inflation along 
$\mathscr{P}(1)$ evidently preserves exact sequences. As $\ell$
does not divide the pro-order of $\mathscr{P}(1)$, the same reasoning
can be complemented with Section I.4.6 
in \cite{VigLivre} to prove exactness of the
restriction.
The next three parts of i) immediately follow from the corresponding
facts for compact induction. 
The last claim about projectives follows from exactness of parahoric
restriction together with part v) of this Proposition and \cite{VigLivre}, 
I.A.1.\\
Part ii) is Claim 4.1.3 of \cite{vigneras2003schur}.\\
Part iii) is Proposition 6.4 in \cite{vigirred}.\\
Part iv) can be checked back e. g. with Section 2.3.1 of
\cite{ollivier2009parabolic}.\\
Part v) is Claim 4.1.2.e in \cite{vigneras2003schur}.
\end{proof}
As a general rule, we will denote representations of a reductive quotient
by $\overline{\rho}, \overline{\sigma}, \ldots$ and the inflations
to a parahoric subgroup by $\rho, \sigma, \ldots$.
\subsection{An application of the Krull-Remak-Schmidt theorem}
\begin{lem}\label{svbjuwebewoigbofgoewibfgoiwehbfg}
Consider two different ways of writing a finite group as
a product of general linear groups over the same field $\mathbb{F}_q$:
\begin{equation*}
\mathfrak{G} = \prod_I \operatorname{GL}_{n_i}(q) =
\prod_J \operatorname{GL}_{m_j}(q)
\end{equation*}
with finite index sets $I$ and $J$. Then there is a bijection
$t: I\rightarrow J$ such that $n_i = m_{t(i)}$.
\end{lem}
\begin{proof}
We start by taking the derived subgroup:
\begin{equation*}
D(\mathfrak{G}) = \prod_I D_{n_i}(q) =
\prod_J D_{m_j}(q),
\end{equation*}
where $D_k(q)$ is short for $D\left(\operatorname{GL}_{k}(q)\right)$.
\begin{prop}
$D_k(q)$ is directly 
indecomposable (i. e. cannot be written in a non-trivial way as a
direct product) for any choice of $k\in \mathbb{N}^{\geq 2}$ and $q$ any
prime power.
\end{prop}
\begin{proof}[Proof of the proposition.]\renewcommand{\qedsymbol}{$\Diamond$}
For $k=2, q=2$, $\operatorname{GL}_{2}(2)$ is
isomorphic to $S_3$, hence $D_2(2)$ is isomorphic to the cyclic group 
with three elements which is indecomposable. For all other values for $k$
and $q$, $D_k(q)$ is isomorphic to $\operatorname{SL}_k(q)$ (\cite{JS}, 
Satz B.3).\\
It is known that $\operatorname{SL}_k(q)$ is quasi-simple 
(see \cite{aschbacher2000finite}, Section
31 and \cite{JS}, Satz B.3, Kor. B.7), hence 
indecomposable (\cite{aschbacher2000finite}, (31.2))
except for $k=2, q=2$ (which was excluded) and
$k=3, q=2$. In the last case, there are (up to isomorphism) four
normal subgroups
and the occurring 
orders are $1, 2, 8$ and $24 = \#(\operatorname{SL}_2(3))$. This is clearly not
compatible with a decomposition of $\operatorname{SL}_2(3)$ as a direct
product.
\end{proof}
For $k=1$, $D_k(q)$ is the trivial group and hence, by definition, not 
directly indecomposable.\vspace{0.3cm}\\
As the next step, we take the centre of $\mathfrak{G}$ and immediately
conclude that $\#(I) = \#(J)$. Define $I'$ as the set of 
all $i\in I$ such that $n_i >1$, analogously for $J'$. Then the proposition 
allows us to apply the Krull-Remak-Schmidt theorem 
(\cite{hungerford1974algebra}, Thm. 3.8) to
\begin{equation*}
D(\mathfrak{G}) = \prod_{I} D_{n_i}(q) =\prod_{I'} D_{n_i}(q) =
\prod_{J'} D_{m_j}(q).
\end{equation*}
Because $D_k(q)$ is not isomorphic to $D_l(q)$ for $k\neq l$,
this tells us that there is a bijection $t: I'\rightarrow J'$ such that
$n_i = m_{t(i)}$ for all $i\in I'$. As $n_i = m_j = 1$ for $i\in I-I', 
j\in J-J'$, $t$ can be continued to be a bijection $I\rightarrow J$
such that $n_i = m_{t(i)}$ for all $i\in I$. As, again, 
$D_k(q)$ is not isomorphic to $D_l(q)$ for $k\neq l$, this implies the claim.
\end{proof}

\subsection{Special instance of the
parahoric Mackey-decomposition}\label{siodhfioebhgoiweigbewboibrfowe}

Notation as follows:
\begin{itemize}
\item $G_m = \operatorname{GL}_m(F)$ and $G = G_n$;
\item $\mathcal{G}_m = \operatorname{GL}_m(q)$ and
$\mathcal{G} = \mathcal{G}_n$;
\item $\mathscr{P}$ is a standard parahoric subgroup of $G$ with the 
property that 
\begin{equation*}
\mathcal{M}=\mathscr{P}/\mathscr{P}(1) = 
\prod_{i=1, \ldots, k} (\mathcal{G}_{n_i})^{m_i}
\end{equation*}
for numbers $n_i, m_i\in \mathbb{N}$ such that $\sum_{i=1}^k n_im_i = n$;
\item $\overline{\rho}$ denotes an $\mathcal{M}$-representation 
of the form $\boxtimes_{i=1}^k (\overline{\rho}_i)^{
m_i}$, where each $\overline{\rho}_i$ is a supercuspidal 
$\mathcal{G}_{n_i}$-representation and $i\neq j$ implies 
$\overline{\rho}_i \not\cong \overline{\rho}_j$ (this is trivially
true if $n_i \neq n_j$);
\item $\mathscr{K} = \operatorname{GL}_n(\mathcal{O})\subset G$.
\end{itemize}
Our aim in this section is to compute $\textit{r}^{G}_{\mathscr{K}}\circ
\textit{i}^{G}_{\mathscr{P}}(\overline{\rho})$, using 
Proposition \ref{sofhiowbeviobweoifbwoeihfewhfowehoigf}.iii, iv.
We say that $d\in D$ survives if $F^{G}_{\mathscr{K}(q)d\mathscr{P}(q)}$
does not vanish on $\overline{\rho}$. As $\overline{\rho}$ 
is cuspidal, $d$ survives precisely
if the the parabolic-restriction-step in the definition of
$F^{G}_{\mathscr{K}(q)d\mathscr{P}(q)}$ is trivial, i. e. if
$\mathscr{P}\cap d^{-1}\mathscr{K}(1)d \subset \mathscr{P}(1)$.
As $\mathscr{K}(1)$ is unaffected by $S_n$-conjugation, this is a condition
on $\lambda$ alone. 
A straight-forward matrix-computation then shows:
\begin{prop}
$d=s\lambda$ survives if and only if $\lambda$ is of the form
\begin{eqnarray*}
\lambda(a_1, \ldots, a_{m}) =\qquad\qquad\qquad\qquad
\qquad\qquad\qquad\qquad\qquad\qquad\qquad\;\;\;\\ 
\operatorname{diag}
(\; \underbrace{\varpi^{a_1}, \ldots, \varpi^{a_1}}_{n_1\text{-times}}, \ldots,
\underbrace{\varpi^{a_{m_1}}, \ldots, \varpi^{a_{m_1}}}_{n_1\text{-times}},
\underbrace{\varpi^{a_{m_1+1}}, \ldots, \varpi^{a_{m_1+1}}}_{n_2\text{-times}},
\ldots \;)
\end{eqnarray*}
with $m =  \sum_{i=1}^k m_i$ and $a_j\in\mathbb{Z}$.
\end{prop}
\begin{proof}\end{proof}
The obvious next step is to understand
the group $f_{\mathscr{K}}(d\mathscr{P}d^{-1} \cap
\mathscr{K})$. Firstly, we write $d= s\lambda$ and remark
\begin{equation*}
f_{\mathscr{K}}((s\lambda)\mathscr{P}(s\lambda)^{-1} \cap
\mathscr{K}) = s [f_{\mathscr{K}}(\lambda\mathscr{P}\lambda^{-1} \cap
\mathscr{K})] s^{-1}
\end{equation*}
because $\mathscr{K}$ is stable under $S_n$-conjugation 
and $f_{\mathscr{K}}$ commutes
with $S_n$-conjugation. Thus we have to concentrate on the shape of the
group $\mathcal{P}_{\lambda}\coloneqq f_{\mathscr{K}}(\lambda\mathscr{P}\lambda^{-1} \cap
\mathscr{K})$ with $\lambda= \lambda(a_1, \ldots, a_{m})$.
\begin{obs}
$\mathcal{P}_{\lambda}$ is a parabolic subgroup of $\mathcal{G}$ with 
Levi-factor
$\mathcal{M} = \mathscr{P}/\mathscr{P}(1)$. It's structure is determined by
the values 
\begin{equation*}
\operatorname{sign}(i,j) = \textbf{1}_{[0, \infty]}(a_i - a_j)
- \textbf{1}_{[- \infty, 0)}(a_i - a_j)
\end{equation*}
for $i\neq j \in \{1, \ldots, m\}$. (If e. g. $a_i \leq a_{i+1}$ for all $i$,
we get the standard parabolic subgroup $\mathcal{P} = f_{\mathscr{K}}(
\mathscr{P})$. If $a_i > a_{i+1}$ for all $i$, we get the opposite
$\overline{\mathcal{P}}$. It is possible to put this into a clumsy formula, but
we don't need that.)
\end{obs}
\begin{proof}
This follows from the definitions and a matrix calculation.
\end{proof}

The next thing to understand is the second step in the definition of
$F^{G}_{\mathscr{K}(q)d\mathscr{P}(q)}$:
\begin{prop}
Conjugation by $d = s\lambda$ from 
$\mathscr{P}\cap d^{-1}\mathscr{K}d$ to
$d\mathscr{P}d^{-1} \cap
\mathscr{K}$ steps down to conjugation by $s$ from the Levi-factor
$\mathcal{M}$ of 
$f_{\mathscr{P}}
(\mathscr{P}\cap d^{-1}\mathscr{K}d)$
to the Levi-factor
$d\mathcal{M}d^{-1} = s\mathcal{M}s^{-1}$
of $f_{\mathscr{K}}(d\mathscr{P}d^{-1} \cap
\mathscr{K})$:
\begin{equation*}
\xymatrix{
\mathscr{P}\cap d^{-1}\mathscr{K}d \ar[r]^{\operatorname{Int}(d)}
\ar[d]_{f_{\mathscr{P}}}&
d\mathscr{P}d^{-1} \cap
\mathscr{K}\ar[d]^{f_{\mathscr{K}}}\\
\mathcal{M} \ar@{=}[d]
& s\mathcal{P}_{\lambda}s^{-1}\ar@{->>}[d]^{
p\coloneqq
\text{mod }s\mathcal{U}_{\lambda}s^{-1}}\\
\mathcal{M}\ar[r]_{\operatorname{Int}(s)} & s\mathcal{M}s^{-1}
}
\end{equation*}
commutes ($\mathcal{U}_{\lambda}$ denotes the unipotent radical of
$\mathcal{P}_{\lambda}$).
\end{prop}
\begin{proof}
This becomes clear as soon as we write out the factorisation 
$\operatorname{Int}(d) = \operatorname{Int}(s)\circ \operatorname{Int}(\lambda)$
. 
Then, as $f_{\mathscr{K}}$ commutes with $\operatorname{Int}(s)$, we have
\begin{equation*}
\xymatrix{
\mathscr{P}\cap d^{-1}\mathscr{K}d \ar@{}[r]|{=} & 
\mathscr{P}\cap \lambda^{-1}\mathscr{K}\lambda\ar[r]^{\operatorname{Int}
(\lambda)}\ar[d]_{f_{\mathscr{P}}}
& \lambda\mathscr{P}\lambda^{-1}\cap \mathscr{K}\ar[d]_{f_{\mathscr{K}}}
\ar[r]^{\operatorname{Int}
(s)}&
s\lambda\mathscr{P}\lambda^{-1}s^{-1}\cap \mathscr{K}
\ar[d]^{f_{\mathscr{K}}}\\
& \mathcal{M}
\ar@{=}[d]
&  \mathcal{P}_{\lambda}\ar@{->>}[d]
& s\mathcal{P}_{\lambda}s^{-1}\ar@{->>}[d]\\
&\mathcal{M} \ar@{.>}[r]_{q}&\mathcal{M} \ar[r]_{\operatorname{Int}(s)}&s\mathcal{M} s^{-1}&
}
\end{equation*}
It is elementary to check that $q$ is the identity.
\end{proof}
We conclude that every summand in the Mackey-decomposition must be 
of the form
\begin{equation*}
\textbf{i}^{\mathcal{G}}_{\mathcal{M}^s \subset \mathcal{Q}^s} 
(\overline{\rho}^s),
\end{equation*}
where $\mathcal{Q}$ is some parabolic subgroup of $\mathcal{G}$ which
admits $\mathcal{M}$ as Levi-factor and $s$ an element 
of $S_n\subset \mathcal{G}$. This is isomorphic
to $\textbf{i}^{\mathcal{G}}_{\mathcal{M} \subset \mathcal{Q}} 
(\overline{\rho})$, hence by the Howlett-Lehrer result
we have
\begin{thm}\label{sdhfohwigobhbweorwgbofegbwr}
\begin{equation*}
\textit{r}^{G}_{\mathscr{K}}\circ
\textit{i}^{G}_{\mathscr{P}}(\overline{\rho}) \cong
\bigoplus_{\text{finite}}
\bf{i}^{\mathcal{G}}_{\mathcal{M} } 
(\overline{\rho}).
\end{equation*}
\end{thm}
\begin{proof}\end{proof}
We need another observation:
\begin{lem}\label{ashfiowehgiwehighweihfioewhriweh}
Let $\mathscr{P}, \mathscr{Q}$ be standard parahoric subgroups
and $\pi$ be a $\mathscr{P}$-rep\-re\-sen\-ta\-tion inflated from a cuspidal
representation $\overline{\pi}$ of $\mathcal{M}=\mathscr{P}/\mathscr{P}(1)$. 
Assume that one summand $F^G_{\mathscr{Q}(q)d
\mathscr{P}(q)}$ in the Mackey-decomposition of 
$\textit{r}^{G}_{\mathscr{Q}}\circ
\textit{i}^{G}_{\mathscr{P}}(\overline{\pi})$
contains a cuspidal representation, say $\overline{\sigma}$, of 
$\mathcal{N} = \mathscr{Q}/\mathscr{Q}(1)$.
Then $\mathscr{Q}=\mathscr{P}^g$ and $\sigma \cong \pi^g$ for
some $g\in G$, in particular
we have
\begin{equation*}
\textit{i}_{\mathscr{P}}^G(\pi) \cong
\textit{i}_{\mathscr{Q}}^G(\sigma) \quad\text{ and }\quad
\bigoplus_{I} 
\begin{bf}i\end{bf}_{\mathcal{M}}^{\mathcal{G}}(\overline{\pi}) \cong
\bigoplus_{I'} 
\begin{bf}i\end{bf}_{\mathcal{N}}^{\mathcal{G}}(\overline{\sigma})
\end{equation*}
for finite index sets $I,I'$.
\end{lem}
\begin{proof}
We will prove this using the language of \cite{vigneras2003schur}, 
Section 4.1.1 and 4.1.4:
In this notation, $\mathscr{P}= P_J$ and $\mathscr{Q}=P_L$ for 
two finite, proper subsets $J, L$ of a fixed basis $\Pi$ for the affine simple
roots defined by $G$ and the standard torus. 
Then $F^G_{\mathscr{Q}(q)d
\mathscr{P}(q)}$ is denoted by $F^d_{L, J}$, where $d$ is now taken from
some distinguished set of representatives $D_{L, J}$ for 
$W_L\backslash W_{\text{aff}}(G)/W_J \cong P_L\backslash G /P_J$.\\
The assumptions of the lemma imply that both 
$\begin{bf}r\end{bf}^{J}_{J\cap d^{-1}L}$ and 
$\begin{bf}i\end{bf}^{L}_{L\cap dJ}$ are trivial, i. e. $J = J\cap d^{-1}L$
and $L = L\cap dJ$. This implies $L = dJ$. As $P_{dJ} = dP_Jd^{-1}$ (4.1.1.(c)),
this implies that $\mathscr{P}$ and $\mathscr{Q}$ are conjugate by $d$.
The remaining middle term $\operatorname{Int}(d)$ in the definition of
$F^d_{L, J}$ is the reduction of the conjugation 
$\operatorname{Int}(d):\mathscr{P}\rightarrow \mathscr{Q}$ by 
4.1.1.(e).
As this maps irreducible representations to irreducible representations,
we see that $\overline{\sigma}$ is actually all of 
$F^d_{L,J}(\overline{\pi})$ and hence 
that $\sigma$ and $\pi$ are $d$-conjugate. The isomorphism
of the parahorically induced representations is therefore established.
The last claim results from applying the functor $\textit{i}^G_{\mathscr{K}}$
and Theorem \ref{sdhfohwigobhbweorwgbofegbwr}.
\end{proof}

\begin{lem}\label{askdhfwbvoiweboifwefhiowegbfiwoe}
Let $\mathcal{P}=\mathcal{M}\mathcal{U}$ and 
$\mathcal{Q}=\mathcal{N}\mathcal{V}$ be two standard parabolic
subgroups of $\mathcal{G}$ and assume that 
\begin{equation*}
\operatorname{Hom}_{\mathcal{G}}
\left(\begin{bf}i\end{bf}_{\mathcal{M}}^{\mathcal{G}}(\pi)
,\begin{bf}i\end{bf}_{\mathcal{N}}^{\mathcal{G}}(\sigma)\right)\neq 0
\end{equation*}
for a cuspidal $\mathcal{M}$-representation $\pi$
and a cuspidal $\mathcal{N}$-representation $\sigma$.
Then $\mathcal{M} = \mathcal{N}^s$ and $\pi \cong \sigma^s$ for some
$s\in S_n$. We can even replace $s$ by some suitable
$t\in S_n$ which acts merely by rearranging the 
blocks in the following sense: Write
\begin{equation*}
\mathcal{M}= \prod_I \mathcal{G}_{n_i}\text{ and } 
\mathcal{N}= \prod_J \mathcal{G}_{m_j}.
\end{equation*}
Define $\Theta^I_k= \{i\in I|n_i = k\}$  and 
$\Theta^J_k= \{j\in J|m_j = k\}$ for $k\in \mathbb{N}$. By Lemma
\ref{svbjuwebewoigbofgoewibfgoiwehbfg},
$\#\Theta^I_k = \#\Theta^J_k$ and we can define $u\coloneqq
\# I = \# J$. Now define the subset
$T\bigl((m_j)_{j\in J}, (n_i)_{i\in I}\bigr) \subset\operatorname{Bij}_{\operatorname{Sets}}(\{1, \ldots, u\},
\{1, \ldots, u\})$
which fulfill $t(\Theta^J_k) = \Theta^I_k$ for all $k$.
There is a canonical way to 
embed $\imath: T\hookrightarrow S_n$ such that conjugation with $\imath(t)$ 
coincides with the map
\begin{equation*}
\mathcal{N}\rightarrow\mathcal{M}\qquad
(x_j)_{j\in J} \mapsto (x_{t(j)})_{j\in J},
\end{equation*}
where each $x_j$ is an element of $\mathcal{G}_{m_j}$.
\end{lem}
\begin{proof}
The first part follows from Frobenius reciprocity and an application 
of the Harish-Chandra Mackey-Theorem (Proposition
\ref{wnogivbweiogbiowebghtnwehgfoibweoigbvoibihnoi} iii).\\
For the second part, fix some 
$t_0 \in T\bigl((m_j)_{j\in J}, (n_i)_{i\in I}\bigr)$.
It is clear $s$ can then be written as $s't_0$, where $s'\in S_n\subset
\mathcal{G}$ normalises $\mathcal{N}$. It follows that $s'$ normalises 
the Young subgroup $S_{\ast} = \prod_I S_{n_i}$ defined by $\mathcal{M}$.
Borevich and Gavron studied the normaliser of $S_{\ast}$ in $S_n$ in 
\cite{borevich1985arrangement},
and following their exposure we can write $s' = s^1s_0$, where 
$s^1\in S_{\ast}$ and $s_0\in T\bigl((n_i)_{i\in i}, (n_i)_{i\in I}\bigr)$.
It is obvious that 
\begin{equation*}
T\bigl((n_i)_{i\in i}, (n_i)_{i\in I}\bigr)\cdot
T\bigl((m_j)_{j\in J}, (n_i)_{i\in I}\bigr)
\subset T\bigl((m_j)_{j\in J}, (n_i)_{i\in I}\bigr),
\end{equation*}
hence $s = s^1s_0t_0$ can be replaced -- up to an $\mathcal{M}$-isomorphism,
which is provided by $s^1$ in this case -- by the rearrangement of
blocks $s_0t_0 \in 
T\bigl((m_j)_{j\in J}, (n_i)_{i\in I}\bigr)$.

\end{proof}

\subsection{Bernstein-decomposition}
\begin{defn}[Level-$0$ representation]
A representation $V$ is called level-$0$ if all its irreducible subquotients
contain a non-zero vector invariant under the maximal open compact subgroup
$\mathscr{K} = \operatorname{GL}_n(\mathcal{O}) \subset G$. $V$ is
positive level if it has no level-$0$ subquotients.
\end{defn}
It is well-known that we can split the category of $G$-representations
as
\begin{equation*}
\mathfrak{R}(G) = \mathfrak{R}^0(G) \oplus \mathfrak{R}^{+}(G), 
\end{equation*}
where $\mathfrak{R}^0(G)$ is the subcategory of level-$0$ representations
and $\mathfrak{R}^{+}(G)$ is the subcategory of representations of
positive level. In the sequel, we will be concerned with level-$0$
representations alone.
\begin{defn}[(Super-)cuspidal pair]\label{sionfeioboiewhbgoihwroehtf}
Let $M\cong \prod_I G_{n_i}\subset G$ be a Levi-subgroup together with 
an $M$-representation $\pi = \boxtimes_I \pi_i$. The pair $(M, \pi)$ is called 
cuspidal (resp. supercuspidal) if each $\pi_i$ is cuspidal 
(resp. supercuspidal). $(M, \pi)$ is called level-$0$ if each $\pi_i$ is
level-$0$. 
Two pairs $(M_1, \pi_1)$ and $(M_2, \pi_2)$
are said to be $G$-equivalent\footnote{Some authors use the term
\textit{inertially equivalent}.} if there is a $g\in G$ and an 
unramified $M_2$-character $\chi$ such that $M_1 = M_2^g$ and
$\pi_1 = (\chi\otimes \pi_2)^g$. This equivalence relation respects
the notions `cuspidal', `supercuspidal' and `level-$0$'.
The generated equivalence-class
is denoted by $[M,\pi]_G$. The set of all equivalence classes of 
level-$0$ supercuspidal types
is called the (level-$0$ supercuspidal)
Bernstein-spectrum $\mathfrak{B}^{0, \text{sc}}(G)$.
\end{defn}
\begin{defn}[Level-$0$ (super-)type]\label{siovnowebogibwoiefhoiwehoifbowie}
Let $\mathscr{P}$ be a parahoric subgroup of $G$ with reductive
quotient $\mathcal{M}= \mathscr{P}/\mathscr{P}(1) \cong 
\prod_J \mathcal{G}_{m_j}$ and a $\mathscr{P}$-representation 
$\rho$ inflated from $\overline{\rho} = \boxtimes_J \overline{\rho}_j$.
The pair $(\mathscr{P}, \rho)$ is called a level-$0$ type (resp. supertype)
if each $\overline{\rho}_j$ is cuspidal (resp. supercuspidal).
Two such (super-)types $(\mathscr{P}, \rho)$ and $(\mathscr{P}', \rho')$
are said to be $G$-equivalent if $\operatorname{ind}_{\mathscr{P}}^G(\rho)
\cong \operatorname{ind}_{\mathscr{P}'}^G(\rho')$. Let us denote
the set of equivalence classes of supertypes by $\mathfrak{S}^0(G)$.
\end{defn}
An $[M, \pi]_G\in \mathfrak{B}^{0, \text{sc}}(G)$ gives rise to a subcategory
\begin{equation*}
\mathfrak{R}^{[M, \pi]_G}(G) \subset \mathfrak{R}^0(G)
\end{equation*}
where $V$ is an object if and only if we can associate to each subquotient
$Q$ of $V$ an $(N, \sigma)\in [M, \pi]_G$ such that $Q$ is isomorphic
to a subquotient of $\textit{i}_{M\subset P}^G(\sigma)$, where $P$ is some
parabolic subgroup of $G$ containing $M$ as Levi-component.

On the other hand, we can associate to the equivalence class defined
by a supertype $(\mathscr{P}, \rho)$ a subcategory
\begin{equation*}
\mathfrak{R}_{(\mathscr{P}, \rho)}(G) \subset \mathfrak{R}^0(G)
\end{equation*}
where $V$ is an object if and only all subquotients of $V$ are subquotients
of $\operatorname{ind}_{\mathscr{P}}^G(\rho)$.
\begin{thm}
\renewcommand{\labelenumi}{\roman{enumi})}
\begin{enumerate}
\item The level-$0$ part of $\mathfrak{R}(G)$ decomposes as
\begin{equation*}
\mathfrak{R}^0(G) = \bigoplus_{[M, \pi]_G\in \mathfrak{B}^{0, \text{sc}}(G)}
\mathfrak{R}^{[M, \pi]_G}(G) 
= \bigoplus_{(\mathscr{P}, \rho)\in\mathfrak{S}^0(G)}
\mathfrak{R}_{(\mathscr{P}, \rho)}(G);
\end{equation*}
\item There exists a bijection between $\mathfrak{B}^{0, \text{sc}}(G)$ and 
$\mathfrak{S}^0(G)$ such that the 
corresponding subcategories are identical and (up to equivalence) 
the indices $I$ and $J$ from Definitions \ref{sionfeioboiewhbgoihwroehtf}
and \ref{siovnowebogibwoiefhoiwehoifbowie} correspond, i. e.
there is a bijection
$t:I{\rightarrow} J$ such that  $n_i = m_{t(i)}$ and $\pi_i \cong \pi_
{i'}$ if and only if 
$\overline{\rho}_{t(i)}\cong \overline{\rho}_{t(i')}$.
\end{enumerate}
\end{thm}
\begin{proof}
Everything is extractable from Chapter IV of \cite{vignéras1998induced}.
\end{proof}
Depending on the structure of the block (or, equivalently, the type) we
distinguish three different sorts of blocks:
\begin{enumerate}
\item A \textit{supercuspidal} block is a block generated by 
a supercuspidal pair of the form $(G, \pi)$. The associated type is
of the form $(\mathscr{K}, \rho)$ with $\rho$ inflated from a supercuspidal
$\mathcal{G}$-representation. All simple objects in this block are unramified
twists of $\pi$.
\item A \textit{simple} block is a block generated by a supercuspidal
pair $(M, \pi)$ with $M= (G_{a})^b$ (with $ab = n$) and
$\pi = \boxtimes_{b\text{ copies }}\pi_0$. The associated type is of
the form $(\mathscr{P}, \rho)$ with $\mathscr{P}/\mathscr{P}(1)
\cong (\mathcal{G}_a)^b$ and $\rho$ inflated from 
$\overline{\rho} = \boxtimes_{b\text{ copies }}\overline{\rho}_0$.
In this case, we also call the supercuspidal pair simple. The
block generated by the pair $(T, R_{\text{triv}})$ with $T\subset G$ the 
standard torus (or, alternatively,
by the type $(\mathscr{I}, R_{\text{triv}})$ with $\mathscr{I}$ the 
Iwahori subgroup) is called the \textit{unipotent} block of $G$. It
contains the trivial $G$-representation.
\item An arbitrary block is called \textit{semisimple}. From now on, 
we always assume
that the associated type is given in the arranged form
described at the beginning of Section \ref{siodhfioebhgoiweigbewboibrfowe}.

\end{enumerate}

\subsection{Surjectivity of Harish-Chandra induction}\label{eu09w7r34r74zutt6569876iujdsgf}
Within Brauer-theory, one constructs a decomposition of
the representation category $\mathfrak{R}(\mathcal{G}_n)$ of 
the finite group $\mathcal{G}_n$, which can be seen as the finite
analogue of Bernstein's decomposition. We collect some implications of this fact
from the literature. We retain the notation from the
preceding sections, in particular from Section
\ref{siodhfioebhgoiweigbewboibrfowe}. Moreover, we define
the subgroup
\begin{equation*}
\mathcal{M}_{\tt{max}} = \prod_{i = 1, \ldots, k} \mathcal{G}_{n_im_i}
\end{equation*}
of $\mathcal{G}$ which contains $\mathcal{M}$ as a Levi-subgroup.
\begin{prop} \label{saifiownfiowbgio}
Let $V$ be an indecomposable 
representation of $\mathcal{M}_{\tt{max}}$ such that
$\operatorname{Hom}_{\mathcal{M}_{\tt{max}}}(V, Q)\neq 0$ for
some subquotient $Q$ of $\bf{i}_{\mathcal{M}}^{\mathcal{M}_{\tt{max}}}(
\overline{\rho})$. Then all irreducible subquotients of 
$V$ are isomorphic to irreducible subquotients of 
$\bf{i}_{\mathcal{M}}^{\mathcal{M}_{\tt{max}}}(
\overline{\rho})$.
\end{prop}
\begin{proof}
At first, let us remark that the claim does not depend on the nature
of $\mathcal{M}_{\tt{max}}$, we could take $\mathcal{G}_n$ instead.\\
It follows from Section 2.4 in \cite{BDC} that $V$ decomposes
as $V = V_0 \oplus V_1$ 
such that each irreducible subquotient of $V_0$ is isomorphic to an irreducible
subquotient of
$\bf{i}_{\mathcal{M}}^{\mathcal{M}_{\tt{max}}}(
\overline{\rho})$ and such that 
this does not happen for any irreducible subquotient
of $V_1$.
As $\operatorname{Hom}_{\mathcal{M}_{\tt{max}}}(V, Q)$ is not empty, 
$V_0$ is not trivial. It follows that $V_1 = 0$.
\end{proof}
The next observation depends crucially on the structure
of $\mathcal{M}_{\tt{max}}$:
\begin{prop}\label{siahfiohwigohewrhigoehwrgtfio}
The Harish-Chandra induction $\bf{i}_{\mathcal{M}_{\tt{max}}}^{\mathcal{G}}$
gives an equivalence of categories between $\mathfrak{R}^{[\mathcal{M},
\overline{\rho}]_{\mathcal{M}_{\tt{max}}}}(\mathcal{M}_{\tt{max}})$ and 
$\mathfrak{R}^{[\mathcal{M},
\overline{\rho}]_{\mathcal{G}}}(\mathcal{G})$. In particular, any irreducible
subquotient of $\bf{i}_{\mathcal{M}}^{\mathcal{G}}(\overline{\rho})$
is of the form $\bf{i}_{\mathcal{M}_{\tt{max}}}^{\mathcal{G}}(\textnormal{X})$ 
for some
irreducible subquotient $X$ of 
$\bf{i}_{\mathcal{M}}^{\mathcal{M}_{\tt{max}}}(\overline{\rho})$.
\end{prop}
\begin{proof}
This is a translation of Lemma 2.4d / Thm. 2.4e in \cite{BDC} into our language.
\end{proof}

\section{Construction of the supercover}\label{sivhnoiwbhi0r3u84ru82734985932}
As described in \cite{VigLivre}, Thm. II.2.4, the irreducible subquotients 
of $\begin{bf}i\end{bf}_{(\mathcal{G}_{n_i})^{m_i}}^{\mathcal{G}_{n_im_i}}
\bigl((\overline{\rho}_i)^{m_i}\bigr)$ may be indexed by partitions
$\mu_i$ of $m_i$ (but it is not clear that different partitions 
give rise to non-isomorphic subquotients).
We denote these subquotients by $P_{i, \mu_i}$, where
$\mu_i$ runs through some subset $\Xi_i$ of the set of partitions
of $m_i$ such that $\mu_i \neq \mu_i'$ implies
$P_{i, \mu_i}\not\cong P_{i, \mu_i'}$.
Define
\begin{equation*}
\Xi = \prod_{i=1, \ldots, k} \Xi_i \quad \text{and}
\quad \Psi = \left\{\boxtimes_{i=1}^{k}
 P_{i, \mu_i}\right\}_{(\mu_1, \ldots, \mu_k)\in \Xi}.
\end{equation*}
Then $\Psi$ is a set of representatives (with respect to the equivalence
relation `isomorphism') for the irreducible subquotients
of $\begin{bf}i\end{bf}_{\mathcal{M}}^{
\mathcal{M}_{\tt{max}}}(\overline{\rho})$.
\begin{prop}\label{shdbgbewrgbwrbigbwr}
Let $X = \boxtimes_{i=1}^{k}
 P_{i, \mu_i}$ be an element of $\Psi$. 
Then each $P_{i, \mu_i}$ admits a projective
cover $\hat{P}_{i, \mu_i}$ and
$\hat{X} = \boxtimes_{i=1}^{k}
 \hat{P}_{i, \mu_i}$ has all its subquotients
in $\Psi$. We denote 
$\hat{\Psi} = \{\hat{X}|X\in \Psi\}$.
\end{prop}
\begin{proof}
The first claim follows from 
\cite{VigLivre}, A.6.b, and the second claim follows from our
Proposition \ref{saifiownfiowbgio}.
\end{proof}
Consider the standard parahoric subgroup $\mathscr{P}_{\tt{max}}$
with finite quotient $\mathcal{M}_{\tt{max}}$ and
denote the inflation of
$\hat{X} \in \hat{\Psi}$ to 
$\mathscr{P}_{\tt{max}}$ by $\tilde{X}$
and the inflation of $X\in\Psi$ by $\tilde{X}^{\ast}$.
Similarly, for 
any $i$ and any $\mu_i$, denote the inflation of 
 $\hat{P}_{i, \mu_i}$ to 
$\mathscr{K}_{n_im_i}=
\operatorname{GL}_{n_im_i}(\mathcal{O})$
by $\tilde{P}_{i, \mu_i}$
and the inflation of 
${P}_{i, \mu_i}$ by $\tilde{P}_{i, \mu_i}^{\ast}$.
\begin{lem}\label{jklsdhfsdhgiuoahgf}
$\tilde{X}$ and $\tilde{P}_{i, \mu_i}$
are finitely generated and projective
(in the category of 
$\mathscr{P}_{\tt{max}}$-representations and
$\mathscr{K}_{n_im_i}$-representations, resp.).
\end{lem}
\begin{proof}
By \cite{VigLivre}, A.6.b, each ${P}_{i, \mu_i}$ is of finite length.
As the property `finitely generated' is passed on to extensions,
`finite length' implies `finitely generated'. Now we
simply have to put together
the proof of Proposition \ref{sofhiowbeviobweoifbwoeihfewhfowehoigf}.i
and
Proposition \ref{snvoiwboivgbeifbgeiwbgvbdwoigvboweobiiojsdbvajb}.ii.
\end{proof}
Indeed, both $\tilde{X}$ and $\tilde{P}_{i, \mu_i}$ are easily seen 
to be cyclic but we don't need this. 
Now we are able to define the protagonist of this paper:
\begin{defn}[Supercover]\label{i3h802hrf8032grt892z4392876432t754t7g}
The pair $(\mathscr{P}_{\tt{max}}, \tilde{\rho})$ with
\begin{equation*}
\tilde{\rho} = \bigoplus_{X\in \Psi}
\tilde{X}
\end{equation*}
is called the supercover of $(\mathscr{P}, \rho)$.
The $\mathscr{P}_{\tt{max}}$-representation 
\begin{equation*}
\tilde{\rho}^{\ast} = \bigoplus_{X\in \Psi}
\tilde{X}^{\ast}
\end{equation*}
is a quotient of $\tilde{\rho}$.
\end{defn}
Similarly, for each $i\in I$ we have a supercover 
$(\mathscr{K}_{n_im_i}, \tilde{\rho}_i)$
of $(\mathscr{P}_i, \rho_i)$, where $\mathscr{P}_i$ is the standard
parahoric subgroup of $G_{n_im_i}$ with reductive quotient
$(\mathcal{G}_{n_i})^{m_i}$ and 
$\tilde{\rho}_i$ is defined by summing over the 
$\tilde{P}_{i, \mu_i}$ with $\mu_i$ running through $\Xi_i$.
\begin{cor}\label{aisfhwebofbweobfuiehuigbjkbjjiozjbasudgifiibvuvu97u89hb}
Both $\tilde{\rho}$ and $\tilde{\rho}_i$ are finite length, finitely generated
and projective.
\end{cor}
\begin{proof}
This follows from Lemma \ref{jklsdhfsdhgiuoahgf}
and Proposition \ref{snvoiwboivgbeifbgeiwbgvbdwoigvboweobiiojsdbvajb}.ii-iii.
\end{proof}


\section{The induced supercover as a progenerator}\label{su0938482HDOFH983r09zhf}
Let $\mathcal{C}$ be some module category. 
\begin{defn}[Progenerator] 
$c\in \operatorname{ob}(\mathcal{C})$ is called a progenerator if
\begin{enumerate}
\item $c$ is projective in $\mathcal{C}$;
\item $c$ is finitely generated;
\item $\operatorname{Hom}_{\mathcal{C}}(c, c')\neq 0$ for any simple 
$c'\in \operatorname{ob}(\mathcal{C})$.
\end{enumerate}
\end{defn}
\begin{thm}
If $c$ is a progenerator, then
\begin{equation*}
\mathcal{C}\;{\cong}\; \operatorname{End}_{\mathcal{C}}(c)\!-\!\textbf{Mod}.
\end{equation*}
\end{thm}
\begin{proof}
Theorem 3.3 in \cite{vigneras2003schur}.
\end{proof}
Let $(M, \pi)$ be the semisimple supercuspidal pair associated
to $(\mathscr{P}, \rho)$, given in the arranged form
$M = \prod_I (G_{n_i})^{m_i}$ and $\pi = \boxtimes_I (\pi_i)^{m_i}$.
\begin{thm}\label{iewoih320h0rth23804u87389435hsdubvuigegt794gsf}
\begin{equation*}
\mathfrak{R}^{[M,\pi]}(G) \;{\cong}\;
\mathscr{H}(G, \mathscr{P}_{\tt{max}}, \tilde{\rho})\!-\!\textbf{Mod}
\end{equation*}
with $\mathscr{H}(G, \mathscr{P}_{\tt{max}}, \tilde{\rho})
= \operatorname{End}_G\bigl(\operatorname{ind}_{\mathscr{P}_{\tt{max}}}^G
(\tilde{\rho})\bigr)$.
\end{thm}
\begin{proof}
We will check that $Y=\operatorname{ind}_{\mathscr{P}_{\tt{max}}}^G
(\tilde{\rho})$ is a progenerator. Firstly, we have to 
remark that it follows from Proposition \ref{shdbgbewrgbwrbigbwr}
that $Y$ indeed lies in $\mathfrak{R}^{[M,\rho]}(G)$. 
Now we can check with the definition:
\begin{enumerate}
\item Projectivity of $Y$ follows from Corollary
\ref{aisfhwebofbweobfuiehuigbjkbjjiozjbasudgifiibvuvu97u89hb} together with
Proposition \ref{aioshoihbfohgbweoifgboiubosdbfoasbf}.iii;
\item The same references tell us that $Y$ is finitely generated;
\item Let $V\in \operatorname{ob}(\mathfrak{R}^{[M,\pi]}(G))$ be irreducible, i. e. 
$V$ appears as a subquotient in
$\textit{i}_{\mathscr{P}}^G(\overline{\rho})$.
As $V$ is level-$0$, it will not vanish upon application of 
$\textit{r}_{\mathscr{K}}^G$. We conclude henceforth from 
Theorem \ref{sdhfohwigobhbweorwgbofegbwr} that 
$\textit{r}_{\mathscr{K}}^G(V)$ contains a non-zero
$\mathcal{G}$-representation which is isomorphic to a subquotient of
$\begin{bf}i\end{bf}_{\mathcal{M}}^{\mathcal{G}}(\overline{\rho})$.
According to Proposition \ref{siahfiohwigohewrhigoehwrgtfio}, this means
that there is some $X\in \Psi$ such that
\begin{equation*}
\operatorname{Hom}_{\mathcal{G}}\bigl(
\begin{bf}i\end{bf}_{\mathcal{M}_{\tt{max}}}^{\mathcal{G}}(X),
\textit{r}_{\mathscr{K}}^G(V)\bigr) \neq 0.
\end{equation*}
Using 
Proposition \ref{sofhiowbeviobweoifbwoeihfewhfowehoigf}.v and
\ref{sofhiowbeviobweoifbwoeihfewhfowehoigf}.ii, 
this implies the existence of a non-zero $G$-map from
$\textit{i}_{\mathscr{P}_{\tt{max}}}^G
(X)$ to $V$, consequently also
from $\operatorname{ind}_{\mathscr{P}_{\tt{max}}}^G
(\tilde{\rho}^{\ast})$ to $V$. This allows us to finish with the 
desired conclusion
\begin{equation*}
\operatorname{Hom}_G\bigl(
\operatorname{ind}_{\mathscr{P}_{\tt{max}}}^G
(\tilde{X}), V\bigr) \neq 0.\qedhere
\end{equation*}
\end{enumerate}
\end{proof}

Of course, this implies at the same time
\begin{equation*}
\mathfrak{R}^{[M_i, \pi_i^{m_i}]}
({G}_{n_im_i}) \;{\cong}\;
\mathscr{H}({G}_{n_im_i}, \mathscr{K}_{n_im_i}, \tilde{\rho}_i)\!-\!
\textbf{Mod}
\end{equation*}
for any $i\in\{1, \ldots, k\}$, where 
\begin{itemize}
\item $M_i = (G_{n_i})^{m_i}$ is a
Levi-subgroup of ${G}_{n_im_i}$;
\item $\mathscr{K}_{n_im_i} = \operatorname{GL}_{n_im_i}(\mathcal{O})$;
\item $[M_i, \pi_i^{m_i}]$ is the simple supercuspidal pair associated to 
the supertype $(\mathscr{K}_{n_im_i}, \tilde{\rho}_i)$.
\end{itemize}

\section{A bound on intertwining}\label{io8378492698456328765987sjdvbifwgtr}

\subsection{Cuspidal and supercuspidal support for $\mathcal{G}_m$}
\label{askfnsbivobbfewbofiweboib}
Consider an irreducible representation $(V,\pi)$ 
of the finite group $\mathcal{G}_m$. We repeat
\begin{defn}[(Super-)cuspidal support]
There exists a cuspidal representation $\sigma$ of some Levi-subgroup
$\mathcal{M}\subset \mathcal{G}$ such that $V$ is a subrepresentation
of $\begin{bf}i\end{bf}_{\mathcal{M}}^{\mathcal{G}}(\sigma)$.
$(\mathcal{M}, \sigma)$ is unique up to $G$-conjugation (see
\cite{VigLivre}, II.2.20) and the $G$-conjugacy
class $[\mathcal{M}, \sigma]^G$ is called the cuspidal support $\mathfrak{Cs}(V)$
of $\pi$.\\
Analogously, there is a supercuspidal representation $\sigma'$ of some
Levi-subgroup $\mathcal{M}'$ such that $\pi$ is a subquotient of 
$\begin{bf}i\end{bf}_{\mathcal{M}'}^{\mathcal{G}}(\sigma')$.
Then $[\mathcal{M}', \sigma']^G$ is called the supercuspidal support 
$\mathfrak{Ss}(V)$ of $\pi$.
\end{defn}
\begin{thm}[Strong conjugacy theorem]\label{asdfnjwbgfbgborgfouwrbogfb}
Let $V$ be irreducible and $(\mathcal{M}, \pi)$, 
$(\mathcal{M}', \pi')\in \mathfrak{Cs}(V)$.
Then 
\begin{equation*}
\begin{bf}i\end{bf}_{\mathcal{M}}^{\mathcal{G}}(\sigma) \cong
\begin{bf}i\end{bf}_{\mathcal{M}'}^{\mathcal{G}}(\sigma')
\end{equation*}
and $\pi$ is both a subrepresentation and a quotient of
$\begin{bf}i\end{bf}_{\mathcal{M}}^{\mathcal{G}}(\sigma)$.
\end{thm}
\begin{proof}
The first part follows from the Howlett-Lehrer result together with
\cite{VigLivre}, I.5.4.iv, and the second part follows
from the existence of the contravariant
duality explained in \cite{BDC}, in particular Corollary 2.2f.
\end{proof}
For the next proposition we use the following notation:
Consider two factorisations $ab = a'b'$ of some number $m$.
This gives rise to two Levi-subgroups $\mathcal{M}=
(\mathcal{G}_a)^b$ and $\mathcal{M}'=
(\mathcal{G}_{a'})^{b'}$ of $\mathcal{G}_m$.
Additionally, fix a supercuspidal $\mathcal{G}_{a}$-representation $\delta$
and a supercuspidal $\mathcal{G}_{a'}$-represen\-ta\-ti\-on $\delta'$.
Let $Q$ be a subquotient of 
$\begin{bf}i\end{bf}_{\mathcal{M}}^{\mathcal{G}_m}(\delta^{b})$ and $Q'$ be
a subquotient of
$\begin{bf}i\end{bf}_{\mathcal{M}'}^{\mathcal{G}_m}(\delta'^{b'})$.
Then $\mathfrak{Cs}(Q)$ contains an element $(M, \sigma)$ where $M$ is
of the form $\prod_{I} \mathcal{G}_{c_i}$ (with $\sum_I c_i = m$)
and $\sigma$ of the form $\boxtimes_I \sigma_i$ with the $\sigma_i$ cuspidal.
Just in the same manner, $\mathfrak{Cs}(Q')$ contains an 
element $(M', \sigma')$ where $M'$ is
of the form $\prod_{J} \mathcal{G}_{d_j}$ (with $\sum_J d_j = m$)
and $\sigma'$ of the form $\boxtimes_J \sigma'_j$ with the $\sigma'_j$ cuspidal.
\begin{prop}\label{aishfiowebhgiwbeiogfbiowebfrbweorwebroiweb}
Assume that $\delta$ is not isomorphic to 
$\delta'$ (which has to be checked only if
the two factorisations are identical). Then, for any choice $i\in I, j\in J$, 
the representations $\sigma_i$ and $\sigma'_j$ are not isomorphic (which, again,
is automatic if $c_i$ does not happen to be equal to $d_j$).
\end{prop}
\begin{proof}
We will use the language and theorems
of \cite{VigLivre}, III.2.3-5. First, $\delta$ can be
written as $\pi(s, a)$, where $s\in 
\overline{\mathbb{F}}_q$ is of degree $a$ 
over $\mathbb{F}_q$.
$Q$ is then of the form $\pi(\delta, \mu) = \pi(I)$, where $\mu$ is a
partition of $b$ and $I$ denotes the $\ell$-t\^{e}te
$(s,a; \mu, b)$. We use Lemma III.2.3.1 in \cite{VigLivre} to associate to $I$
a certain $\ell$-pied sp\'{e}cial $J = ((s, a_j; \mu_j, b_j))_j$ such 
that $\pi(I)=\pi(J)$ (see Thm. III.2.5 in \cite{VigLivre}), where it is important that
we extract from the proof of Lemma II.2.3.1 that $s$ indeed equals the one
used in the definition of $I$. According to Cor. III.2.5, we can 
take as representative of $\mathfrak{Cs}(Q)$
\begin{equation*}
M = \prod_j (\mathcal{G}_{a_j})^{b_j} \text{ and } 
\sigma = \boxtimes_j \pi(s, a_j)^{b_j}
\end{equation*}
We can do the same things with $Q'$ and conclude
\begin{equation*}
M' = \prod_{j'} (\mathcal{G}_{a'_{j'}})^{b'_{j'}} \text{ and } 
\sigma' = \boxtimes_{j'} \pi(s', a'_{j'})^{b'_{j'}},
\end{equation*}
where it is critical that $s$ and $s'$ are not associated (i. e. 
$s \neq \tau(s)$ for all $\tau\in \operatorname{Gal}(
\overline{\mathbb{F}}_q / {\mathbb{F}}_q)$, see section III.2.2 in 
\cite{VigLivre})
because $\delta$ and $\delta'$ are not isomorphic. This clearly 
implies that no factor of $\sigma'$ can be isomorphic to any 
factor of $\sigma$.
\end{proof}

\subsection{Generalities on intertwining}\label{asfhoiwgiowebiowehfiowehfio}
Consider two parahoric subgroups 
$\mathscr{Q}, \mathscr{P}$ of $G$, a 
$\mathscr{Q}$-representation 
$\kappa$ (inflated from 
a representation $\overline{\kappa}$
of $\mathcal{M}=\mathscr{Q}/\mathscr{Q}(1)$)
and a $\mathscr{P}$-representation 
$\pi$ (inflated from 
a representation $\overline{\pi}$
of $\mathcal{N}=\mathscr{P}/\mathscr{P}(1)$).
Consider the space
\begin{equation*}
M_G(\kappa, \pi) = 
\operatorname{Hom}_G\left( 
\textit{i}_{\mathscr{Q}}^G(\overline{\kappa}),
\textit{i}_{\mathscr{P}}^G(\overline{\pi})\right).
\end{equation*}
Using Proposition \ref{sofhiowbeviobweoifbwoeihfewhfowehoigf}.v (and sticking
to the notation of \cite{VigLivre}, Chapter I.8.5), 
we can write this as the model
\begin{equation*}
M_G'(\kappa, \pi) = 
\operatorname{Hom}_{\mathcal{M}}\left( 
\overline{\kappa},
\bigoplus_{g\in \mathscr{Q}\backslash G/
\mathscr{P}}
F^G_{\mathscr{Q}(q)g\mathscr{P}(q)}
(\overline{\pi})
\right).
\end{equation*}
Another model is the space
$M_G''(\kappa, \pi)$ consisting of maps
$f: G\rightarrow \operatorname{Hom}_R(V_{\pi}, V_{\kappa})$
which fulfill
\begin{itemize}
\item $f(pxq) = \pi(p)\circ f(x) \circ \kappa(q)$ for all $p\in \mathscr{P},
x\in G, q\in \mathscr{Q}$;
\item $f$ is supported on finitely many double cosets 
$\mathscr{P}x\mathscr{Q}$;
\end{itemize} 
(where $V_{\pi}$ denotes the underlying space of $\pi$).
\begin{defn}[Intertwining]
The Intertwining set $\mathcal{I}_G(\kappa, \pi)\subset 
\mathscr{P}\backslash G/\mathscr{Q}$ is the set of all double 
cosets $\mathscr{P}x\mathscr{Q}$ for which there exists an $f\in
M_G''(\kappa, \pi)$ such that $f(x)\neq 0$. 
\end{defn}
\begin{obs}
It follows readily from the explanations in I.8.5 in \cite{VigLivre} that
$\mathscr{P}x\mathscr{Q} \in \mathcal{I}_G(\kappa, \pi)$
if and only if there is a map in $ 
M_G'(\kappa, \pi)$ which has non-zero contribution to the 
$\mathscr{Q}x^{-1}\mathscr{P}$-th summand. This is the case if and only
if 
\begin{equation*}
\operatorname{Hom}_{\mathscr{Q}\cap x\mathscr{P}x^{-1}} (\kappa, \pi^x) \neq 0.
\end{equation*}

\end{obs}
\begin{rem}
Inspired from this observation, 
we will say also that $\mathscr{Q}x^{-1}\mathscr{P}$ is in the intertwining
set. This can lead to confusion only if $\mathscr{P} = 
\mathscr{Q}$ and $\pi\neq \kappa$, and we will not encounter
this situation in the sequel. Remark that we could have a more
uniform notation of this if we used a version of Mackey's
decomposition which sums over $\mathscr{P}\backslash G/
\mathscr{Q}$ instead of $\mathscr{Q}\backslash G/
\mathscr{P}$ (and indeed this is preferred by some authors).
\end{rem}
\subsection{Intertwining of two subquotients}
We introduce the Levi-subgroup 
\begin{equation*}
M_{\tt{max}} = \prod_{i=1}^k G_{a_ib_i}\subset G.
\end{equation*}
\begin{thm}\label{shvbweogfbowehofwehofhgowehgfowbroubhwoegboiwe}
Let $X=\boxtimes_{i=1}^k X_i$ and $Y=\boxtimes_{i=1}^k Y_i$ be two
elements of $\Psi$.
Then 
\begin{equation*}
\operatorname{Hom}_G\left(\textit{i}_{\mathcal{M}_{\tt{max}}}^G(X),
\textit{i}_{\mathcal{M}_{\tt{max}}}^G(Y)\right)
\end{equation*}
has intertwining contained in 
$\mathscr{P}_{\tt{max}}M_{\tt{max}}\mathscr{P}_{\tt{max}}$.
\end{thm}
\begin{proof}
\underline{Step 1}: In accordance with Section \ref{askfnsbivobbfewbofiweboib},
can find two cuspidal representations $\pi_1, \pi_2$ of (standard)
Levi-subgroups
$\mathcal{M}_1,\mathcal{M}_2\subset \mathcal{M}_{\tt{max}}$ such that there
are maps
\begin{equation*}
p:\begin{bf}i\end{bf}_{\mathcal{M}_1}^{\mathcal{M}_{\tt{max}}}(\pi_1)
\twoheadrightarrow X\quad \text{and}\quad 
\imath: Y\hookrightarrow 
\begin{bf}i\end{bf}_{\mathcal{M}_2}^{\mathcal{M}_{\tt{max}}}(\pi_2).
\end{equation*}
\underline{Step 2}: Define a map
\begin{equation*}
\Theta: 
\operatorname{Hom}_G\left(\textit{i}_{\mathcal{M}_{\tt{max}}}^G(X),
\textit{i}_{\mathcal{M}_{\tt{max}}}^G(Y)\right)
\rightarrow
\operatorname{Hom}_G\left(\textit{i}_{\mathcal{M}_{1}}^G(\pi_1),
\textit{i}_{\mathcal{M}_{2}}^G(\pi_2)\right)
\end{equation*}
by sending $\varphi$ 
to $\textit{i}_{\mathcal{M}_{\tt{max}}}^G(\imath)\circ \varphi\circ
\textit{i}_{\mathcal{M}_{\tt{max}}}^G(p)$. By translation 
into the model $M''$ of Section
\ref{asfhoiwgiowebiowehfiowehfio}, $\Theta$ can be understood as the map
\begin{equation*}
M_G''(X, Y)\rightarrow 
M_G''\left( \begin{bf}i\end{bf}_{\mathcal{M}_1}^{\mathcal{M}_{\tt{max}}}(\pi_1),
\begin{bf}i\end{bf}_{\mathcal{M}_2}^{\mathcal{M}_{\tt{max}}}(\pi_2)\right)
\text{ defined by } \Theta(f): g\mapsto 
\imath\circ f(g) \circ p.
\end{equation*}
$\Theta$ respects the support in the sense that $f(g)\neq 0$ implies
$\Theta(f)(g)\neq 0$.\newline
\underline{Step 3}: By transitivity of the parahoric induction, it is 
clear that 
\begin{equation*}
M_G''\left( \begin{bf}i\end{bf}_{\mathcal{M}_1}^{\mathcal{M}_{\tt{max}}}(\pi_1),
\begin{bf}i\end{bf}_{\mathcal{M}_2}^{\mathcal{M}_{\tt{max}}}(\pi_2)\right)
\cong 
M_G''(\pi_1, \pi_2).
\end{equation*}
This identity is compatible with the intertwining set in the following sense:
\begin{prop}
Define $\mathscr{P}_i$ to be the standard parahoric subgroup in $G$ with
the property $\mathscr{P}_i/\mathscr{P}_i(1) = \mathcal{M}_i$ ($i=1,2$).\\
Let $g_0 \in G$ and assume that $\mathscr{P}_1 kg_0k'\mathscr{P}_2$
is not in the intertwining of $M''(\pi_1,\pi_2)$ for all choices 
$k,k'\in \mathscr{P}_{\tt{max}}$. Then $\mathscr{P}_{\tt{max}}g_0
\mathscr{P}_{\tt{max}}$ is not in the intertwining of
$M_G''\left( \begin{bf}i\end{bf}_{\mathcal{M}_1}^{\mathcal{M}_{\tt{max}}}(\pi_1),
\begin{bf}i\end{bf}_{\mathcal{M}_2}^{\mathcal{M}_{\tt{max}}}(\pi_2)\right)
$.
\end{prop}
\begin{proof}[Proof of the proposition.]\renewcommand{\qedsymbol}{$\Diamond$}
Assume, that $\mathscr{P}_{\tt{max}}g_0
\mathscr{P}_{\tt{max}}$ is in the intertwining of
$M_G''\left( \begin{bf}i\end{bf}_{\mathcal{M}_1}^{\mathcal{M}_{\tt{max}}}(\pi_1),
\begin{bf}i\end{bf}_{\mathcal{M}_2}^{\mathcal{M}_{\tt{max}}}(\pi_2)\right)
$. Then there is an $f$ in this set for which 
we can fix a $g\in \mathscr{P}_{\tt{max}}g_0
\mathscr{P}_{\tt{max}}$ and write $f(g)(\zeta_1) = \zeta_2 + r$ for suitable
choices of
\begin{itemize}
\item $\gamma_i \in 
\mathcal{M}_{\tt{max}}$ ($i=1,2$);
\item $\zeta_i \in \begin{bf}i\end{bf}_{\mathcal{M}_i}^{\mathcal{M}_{
\tt{max}}}(\pi_i)$ non-zero with support $\mathscr{P}_i\cdot \gamma_i$
($i=1,2$);
\item $r\in \begin{bf}i\end{bf}_{\mathcal{M}_2}^{\mathcal{M}_{
\tt{max}}}(\pi_2)$ supported outside $\mathscr{P}_2\cdot \gamma_2$.
\end{itemize}
By replacing $g$ by $\tilde{\gamma}_2^{-1}g\tilde{\gamma}_1^{-1}$ (where
$\tilde{\gamma}_i$ denotes a lift of $\gamma_i$ to 
$\mathscr{P}_{\tt{max}}$), we can assume that $\gamma_1 = \gamma_2 =1$.
In order to prove the claim, we have to construct a non-zero map 
$\varepsilon \in M_G''(\pi_1, \pi_2)$ with support 
$\mathscr{P}_2g\mathscr{P}_1$. This is done as follows: We have two maps
\begin{equation*}
s: V_{\pi_1}\hookrightarrow V_{
\begin{bf}i\end{bf}_{\mathcal{M}_1}^{\mathcal{M}_{\tt{max}}}(\pi_1)} \qquad
t: V_{
\begin{bf}i\end{bf}_{\mathcal{M}_2}^{\mathcal{M}_{\tt{max}}}(\pi_2)}
\twoheadrightarrow V_{\pi_2}
\end{equation*}
(where, as usual, $V_\pi$ denotes the space underlying $\pi$) given by
$s(v) = \xi_v$, where $\xi_v$ has support $\mathscr{P}_1$ and maps 
$p$ to $pv$. $t$ is defined by sending a $\xi$ to $\xi(1)$. $s$ intertwines
with $\mathscr{P}_1$ and $t$ intertwines with $\mathscr{P}_2$.
Hence we can define the 
map $\varepsilon:G\rightarrow \operatorname{Hom}_G(\pi_1, \pi_2)$
supported on $\mathscr{P}_2g\mathscr{P}_1$ and characterised by sending 
$p_2gp_1$ to 
\begin{equation*}
t\circ f(p_2gp_1) \circ s = t\circ
\begin{bf}i\end{bf}_{\mathcal{M}_2}^{\mathcal{M}_{\tt{max}}}(\pi_2) (p_2)\circ
f(g) \circ 
\begin{bf}i\end{bf}_{\mathcal{M}_1}^{\mathcal{M}_{\tt{max}}}(\pi_1) (p_1)\circ
s = \pi_2(p_2)\circ t\circ f(g)\circ s \circ \pi_1(p_1).
\end{equation*}
\end{proof}\noindent
\underline{Step 4}: The strategy now is to show that the intertwining
of $M_G(\pi_1, \pi_2)$ lies inside $\mathscr{P}_2M_{\tt{max}}\mathscr{P}_1$.
As a consequence of the proposition, this will imply the claim. For this, we
use the model $M'$ and write
\begin{equation*}
\operatorname{Hom}_{\mathcal{M}_1}\left( \pi_1, 
\bigoplus_{d\in D} F^G_{\mathscr{P}_1(q)d\mathscr{P}_2(q)}(\pi_2)\right).
\end{equation*}
If there is no $d$ such that there is an $f$ in $M'$ with contribution
to the $d$th summand, there is nothing to prove. If there is such 
a $d$, we can apply Lemma \ref{ashfiowehgiwehighweihfioewhriweh} which
tells us that
\begin{equation*}
\bigoplus_J
\begin{bf}i\end{bf}_{\mathcal{M}_1}^{\mathcal{G}}(\pi_1)
\cong
\bigoplus_{J'}
\begin{bf}i\end{bf}_{\mathcal{M}_2}^{\mathcal{G}}(\pi_2)
\end{equation*}
for two finite index sets $J, J'$. This, in turn, puts us in a position
to apply Lemma \ref{askdhfwbvoiweboifwefhiowegbfiwoe} and conclude
that -- up to $\mathcal{M}_1$-isomorphism -- 
$(\mathcal{M}_1, \pi_1)$ and $(\mathcal{M}_2, \pi_2)$ are conjugated
by a simple rearrangement of blocks $t\in S_n\subset
\mathcal{G}$ (as in the 
formulation of Lemma \ref{askdhfwbvoiweboifwefhiowegbfiwoe}).
Proposition \ref{aishfiowebhgiwbeiogfbiowebfrbweorwebroiweb} then tells us
that $t$ must be contained in $\mathcal{M}_{\tt{max}}$.
We conclude 
\begin{equation*}
\begin{bf}i\end{bf}_{\mathcal{M}_1}^{\mathcal{M}_{\tt{max}}}(\pi_1)
\cong
\begin{bf}i\end{bf}_{\mathcal{M}_2}^{\mathcal{M}_{\tt{max}}}(\pi_2),
\end{equation*}
i. e. $Y$ is a quotient of 
$\begin{bf}i\end{bf}_{\mathcal{M}_1}^{\mathcal{M}_{\tt{max}}}(\pi_1)$
and we actually have to compute the intertwining of $\mathscr{H}(G, 
\mathscr{P}_1, \pi_1)$ instead of $M_G(\pi_1, \pi_2)$. It is known
that this intertwining is contained in $\mathscr{P}_1 M_{\tt{max}}\mathscr{P}_1$
(see \cite{vignéras1998induced}, Section IV.3.2-3).
\end{proof} 

\subsection{Bounds for intertwining pass over to extensions}\label{aohfoweiht08zr84z3r849tgr94g0}
First, we need a general 
\begin{lem}
In the module-category over some ring $\mathcal{R}$, 
consider two short exact sequences
\begin{equation*}
0 \rightarrow A \overset{a}{\rightarrow} B 
\overset{b}{\rightarrow} C \rightarrow 0 \qquad
\text{and} \qquad 
0 \rightarrow X \overset{x}{\rightarrow} Y 
\overset{y}{\rightarrow} Z \rightarrow 0.
\end{equation*}
If $\operatorname{Hom}_{\mathcal{R}}(U, V)=0$ for all $U\in \{A, C\}, 
V\in \{X, Z\}$, then $\operatorname{Hom}_{\mathcal{R}}(B, Y)=0$.
\end{lem}
\begin{proof}
Assume we have a non-zero $f\in \operatorname{Hom}_{\mathcal{R}}(B, Y)$.\\
The first observation is that $f\circ a = 0$: Assume, this is not the case, 
i. e. there is an $\alpha \in A$ such that $f(a(\alpha))\neq 0$. Then in any
case $y(f(a(\alpha)))$ must vanish, because otherwise we would have produced
a non-zero arrow $A\rightarrow Z$. So $f(a(\alpha))$ lies in the image 
of $x$. This is true for any $\alpha'$ (with $f(a(\alpha'))$ zero or not), 
hence $f\circ a$ restricts
to a map $A\rightarrow \operatorname{im}(x) \cong X$. By the 
assumption on $\alpha$ this map is non-zero and this is a contradiction.\\
So now we can talk about the following diagram
\begin{equation*}
\xymatrix{
0 \ar[r] &A \ar[r]^{a}\ar[d]_{0} &B\ar[r]^{b}\ar[d]^{f} &C\ar[r]
\ar@{.>}[d]^{g} &0\\
0 \ar[r] &X \ar[r]_{x} &Y\ar[r]_{y} &Z\ar[r] &0
}
\end{equation*}
and the $ABXY$-square is commuting. We will now construct a 
$g$ making $BCYZ$ commute:\\
Let $\gamma \in C$. Then we can take a pre-image $b^{-1}(\gamma)$ and
consider $f(b^{-1}(\gamma))$. The fact that the left square commutes 
implies that this element in $Y$ is independent of the choice of 
the pre-image. Then define $g(\gamma)$ as $y(f(b^{-1}(\gamma)))$. It is 
straightforward to see that this assignment is $\mathcal{R}$-equivariant.
By construction, the $BCYZ$-square commutes. By assumption, $g = 0$. Hence
we are in the situation of the following commuting diagram
\begin{equation*}
\xymatrix{
0 \ar[r] &A \ar[r]^{a}\ar[d]_{0} &B\ar[r]^{b}\ar[d]^{f} &C\ar[r]
\ar[d]^{0} &0\\
0 \ar[r] &X \ar[r]_{x} &Y\ar[r]_{y} &Z\ar[r] &0
}
\end{equation*}
and we still assume $f\neq 0$.\\
Return to the assignment $\gamma \mapsto f(b^{-1}(\gamma))$. This indeed
is a well-defined $\mathcal{R}$-homomorphism $h:C\rightarrow Y$.
As $f$ is assumed to be non-zero, so is $h$. Moreover, 
$\operatorname{im}(h) = \operatorname{im}(f)$. But 
$\operatorname{im}(f) \subset \operatorname{ker}(y) =
\operatorname{im}(x)\cong X$. Hence we have produced a non-zero map
$C\rightarrow X$ which gives the final 
contradiction.
\end{proof}
Now let $M$ be a (finite-length) module and $\mathcal{Q}$ some 
set of modules.
We say that $M$ decomposes into $\mathcal{Q}$ if there is a sequence of nested
submodules
\begin{equation*}
M_1 \subset M_2 \subset \ldots \subset M_n = M
\end{equation*}
with $M_1$ and all the $M_i/M_{i-1}$ for $2\leq q \leq n$ 
being isomorphic to members of $\mathcal{Q}$
(and we
remark that we do not assume that all members of $\mathcal{Q}$ are
irreducible). Define the number $\operatorname{length}_{\mathcal{Q}}(M)$ 
to be the smallest $n\in\mathbb{N}$ 
such that a nested sequence as above
exists. The zero-module has $\mathcal{Q}$-length $0$ for any $\mathcal{Q}$.
If $\mathcal{V}$ is another set of modules, we write $\operatorname{Hom}(
\mathcal{Q}, \mathcal{V}) = 0$ if
$\operatorname{Hom}({Q}, {V}) = 0$ for any choice $Q\in\mathcal{Q},
V\in\mathcal{V}$. 
\begin{cor} \label{asdfhjkashfjaskdhf}
Let $M, N$ in $\mathcal{R}\!-\!\textbf{Mod}$ be finite-length
such that $M$ decomposes into
$\mathcal{Q}$ and $N$ into $\mathcal{V}$. 
Then $\operatorname{Hom}(\mathcal{Q}, \mathcal{V})= 0$ implies
$\operatorname{Hom}(M, N)= 0$.
\end{cor}
\begin{proof}
Induction on $d_{M, N} = \operatorname{max}\{
\operatorname{length}_{\mathcal{Q}}(M),
\operatorname{length}_{\mathcal{V}}(N)\}$: If $d_{M, N}=0$, the statement
is obviously true; So let $d_{M, N}>1$ and assume the statement is
known for all $M', N'$ with $d_{M', N'}<d_{M, N}$.
(Additionally, assume that neither $M$ nor $N$ are zero; in that case 
the claim is true anyway.)
Then we can embed $M$ and $N$ into sequences
\begin{equation*}
0 \rightarrow M' \rightarrow M\rightarrow Q \rightarrow 0 \qquad \text{and}
\qquad
0 \rightarrow N' \rightarrow N\rightarrow V \rightarrow 0
\end{equation*}
where $Q\in \mathcal{Q}, V\in\mathcal{V}$, $M'$ decomposes into $\mathcal{Q}$, 
$N'$ decomposes into $\mathcal{V}$ and
$d_{M', N'}<d_{M, N}$. The proof now follows from the above lemma.
\end{proof}
We can use this machinery to prove
\begin{thm}\label{sdoihfoiwhbgfiwebhfgbwei}
The super-Hecke algebra $\mathscr{H}(G, \mathscr{P}_{\tt{max}}, \tilde{\rho})$
has intertwining contained in $\mathscr{P}_{\tt{max}}M_{\tt{max}}
\mathscr{P}_{\tt{max}}$.
\end{thm}
\begin{proof}
Let $g\in G - \mathscr{P}_{\tt{max}}M_{\tt{max}}
\mathscr{P}_{\tt{max}}$. The $\mathscr{P}_{\tt{max}}
\cap g\mathscr{P}_{\tt{max}}g^{-1}$-representation \\
$\operatorname{ind}_{\mathscr{P}_{\tt{max}}}^G(\tilde{\rho})$ decomposes
into $\mathcal{Q}= \{\operatorname{ind}_{\mathscr{P}_{\tt{max}}}^G(X)|X\in \Psi\}$
and the $\mathscr{P}_{\tt{max}}
\cap g\mathscr{P}_{\tt{max}}g^{-1}$-representation
$\operatorname{ind}_{\mathscr{P}_{\tt{max}}}^G(\tilde{\rho})^g$ decomposes
into $\mathcal{Q}^g= \{\operatorname{ind}_{\mathscr{P}_{\tt{max}}}^G(X)^g
|X\in \Psi\}$. By Theorem
\ref{shvbweogfbowehofwehofhgowehgfowbroubhwoegboiwe}, 
\begin{equation*}
\operatorname{Hom}_{\mathscr{P}_{\tt{max}}
\cap g\mathscr{P}_{\tt{max}}g^{-1}}(\mathcal{Q},\mathcal{Q}^g)=0.
\end{equation*}
The claim therefore follows from 
Corollary \ref{asdfhjkashfjaskdhf}.
\end{proof}

\section{Factorisation of the Hecke algebra of a supercover}\label{aoihf0eh803t90uwegf8aiohoihg89327zrgt8nh40j8}
In the last section, we showed (Theorem \ref{sdoihfoiwhbgfiwebhfgbwei})
that the subspace 
\begin{equation*}
\mathscr{H}(\mathscr{P}_{\tt{max}}M_{\tt{max}}
\mathscr{P}_{\tt{max}}, \mathscr{P}_{\tt{max}}, \tilde{\rho}) \subset
\mathscr{H}(G, \mathscr{P}_{\tt{max}}, \tilde{\rho})
\end{equation*}
of functions with support in $\mathscr{P}_{\tt{max}}M_{\tt{max}}
\mathscr{P}_{\tt{max}}$ is actually all of
$\mathscr{H}(G, \mathscr{P}_{\tt{max}}, \tilde{\rho})$, hence 
it is an $R$-algebra.
This allows us to use Proposition II.8 and Proposition II.4 of 
\cite{vignéras1998induced}, which tells us that there is an isomorphism of
algebras
\begin{equation*}
\mathscr{H}(G, \mathscr{P}_{\tt{max}}, \tilde{\rho}) \cong
\mathscr{H}(M_{\tt{max}}, \mathscr{P}_{\tt{max}}^{\circ}, 
\tilde{\rho}^{\circ}),
\end{equation*}
where $\mathscr{P}_{\tt{max}}^{\circ} = \mathscr{P}_{\tt{max}}\cap M_{\tt{max}}$
and $\tilde{\rho}^{\circ} = 
\tilde{\rho}|\mathscr{P}_{\tt{max}}^{\circ}$. Using this, we can show
\begin{thm} \label{saoihfoih89thg2984gt98234rgh92trg5942rvgb9uvb9q87gwtgb4}
There are numbers $u_i\in \mathbb{N}$ such that
\begin{equation*}
\mathscr{H}(G, \mathscr{P}_{\tt{max}}, \tilde{\rho}) \cong
\bigotimes_{i\in I} \mathscr{H}\bigl(G_{n_im_i}, \mathscr{K}_{n_im_i}, 
\oplus_{u_i\text{copies }}\,\tilde{\rho}_i\bigr).
\end{equation*}
The $i$-th tensor factor is Morita-equivalent to
\begin{equation*}
\mathscr{H}(G_{n_im_i}, \mathscr{K}_{n_im_i}, 
\tilde{\rho}_i).
\end{equation*}
\end{thm}
\begin{proof}
Let's unravel the definitions:
\begin{itemize}
\item $M_{\tt{max}} = \prod_I G_{n_im_i}$;
\item $\mathscr{P}_{\tt{max}}^{\circ} = \prod_I {\mathscr{K}}_{n_im_i}$;
\item $\tilde{\rho}^{\circ} = \boxtimes_I \bigl( 
\oplus_{u_i \text{copies }} \tilde{\rho}_i\bigr)$ with
$u_i = \frac{\# \Xi}{\# \Xi_i}$.
\end{itemize}
The first claim follows now from 
applying 
Proposition \ref{wihfgoiwbgwighni4ehbio4hbgoihno2i4hng42}
followed by 
Lemma \ref{eiheg509hw8hfw2h4gf8h2498eghf84e29ghf8924htf}.
For the second claim, we see that the $i$-th factor is equal to
\begin{equation*}
\operatorname{End}_{G_{n_im_i}}\Bigl( \oplus_{u_i \text{copies }}
\operatorname{ind}_{\mathscr{K}_{n_im_i}}^{G_{n_im_i}}
(\tilde{\rho}_i)\Bigr)	
\cong \mathbb{M}_{u_i\times u_i} \Bigl(
\mathscr{H}(G_{n_im_i}, \mathscr{K}_{n_im_i}, 
\tilde{\rho}_i)\Bigr)
\end{equation*}
and this ring is Morita-equivalent to 
$\mathscr{H}(G_{n_im_i}, \mathscr{K}_{n_im_i}, 
\tilde{\rho}_i)$.
\end{proof}

\section{A worked example} \label{ewihfih08ghtf832gbtfug32b984g732g4rq09t}
We conclude with working out the example
\begin{itemize}
\item $G= \operatorname{GL}_2(\mathbb{Q}_p)$;
\item $M = T$ and $\pi = \pi_1 \boxtimes \pi_2$, where $\pi_i$ are
level-$0$ characters such that $\pi_1 / \pi_2$ is ramified;
\item Let $\chi_i$ be the restriction of $\pi_i$ to $\mathbb{Z}_p^{\times}$.
Then $\chi_i$ is inflated from a character 
$\overline{\chi}_i$ of
$\mathbb{Z}_p^{\times}/(1+\mathfrak{P})
\cong k^{\times}$. The associated type to $(M, \pi)$ is $(I, \chi)$,
where $I$ is the Iwahori-subgroup of $G$ and $\chi$ is inflated from
$\overline{\chi}_1\boxtimes \overline{\chi}_2$.

\end{itemize}
Decompose $s = k^{\times}$ as $s = s_{\ell}\times s^{(\ell)}$, where $
s_{\ell}$ is an $\ell$-group and the order of 
$s^{(\ell)}$ is prime to $\ell$. 
We also set $\mathcal{T}= s\times s = \mathcal{T}_{\ell} \times
\mathcal{T}^{(\ell)} = (s_{\ell}\times s_{\ell}) \times (s^{(\ell)}
\times s^{(\ell)})$. If $\theta$ is the projection $\mathbb{Z}_p^{\times}\twoheadrightarrow 
k^{\times}$,
denote by $(\mathbb{Z}_p^{\times})^{(\ell)}$ the pre-image of $s^{(\ell)}$
under $\theta$. This gives rise to a subgroup
\begin{equation*}
I^{(\ell)} = \begin{pmatrix}
(\mathbb{Z}_p^{\times})^{(\ell)} & \mathbb{Z}_p \\
\mathfrak{P} & (\mathbb{Z}_p^{\times})^{(\ell)}
\end{pmatrix}
\end{equation*}
of $I$. We have
\begin{itemize}
\item $I/I(1) = \mathcal{T}$;
\item $I^{(\ell)}/I(1) = \mathcal{T}^{(\ell)}$;
\item $I/I^{(\ell)} = \mathcal{T}_{\ell}$.
\end{itemize}
Inflation among $I(1)$ defines two functors
\begin{eqnarray*}
\operatorname{infl}_{\mathcal{T}}^{I}: 
\mathcal{T}\!-\!\textbf{Mod}\rightarrow I\!-\!\textbf{Mod};\quad\\
\operatorname{infl}_{\mathcal{T}^{(\ell)}}^{I^{(\ell)}}:
\mathcal{T}^{(\ell)}\!-\!\textbf{Mod}
\rightarrow I^{(\ell)}\!-\!\textbf{Mod}.
\end{eqnarray*}

\begin{prop}
The functors $\operatorname{ind}_{I^{(\ell)}}^I\circ 
\operatorname{infl}_{\mathcal{T}^{(\ell)}}^{I^{(\ell)}}$ and
$\operatorname{infl}_{\mathcal{T}}^{I} \circ
\operatorname{ind}_{\mathcal{T}^{(\ell)}}^{\mathcal{T}}$ are isomorphic.
\end{prop}
\begin{proof}
Let $(\pi, V)$ be in $\mathcal{T}^{(\ell)}\!-\!\textbf{Mod}$.
$\operatorname{ind}_{I^{(\ell)}}^I\circ 
\operatorname{infl}_{\mathcal{T}^{(\ell)}}^{I^{(\ell)}}(V)$
consists of all maps
\begin{equation*}
f: I\rightarrow V \text{ such that } f(i^{(\ell)}i) = 
\pi(\overline{i^{(\ell)}})f(i)\text{ for all }i^{(\ell)}\in 
I^{(\ell)}, i\in I,
\end{equation*}
where $\overline{i^{(\ell)}} = \theta(i^{(\ell)})$ and where $i\in I$
acts by $f\mapsto f(\textvisiblespace\,i)$.
$\operatorname{infl}_{\mathcal{T}}^{I} \circ
\operatorname{ind}_{\mathcal{T}^{(\ell)}}^{\mathcal{T}}(V)$ on the other
hand consists of all maps
\begin{equation*}
\varphi: \mathcal{T}\rightarrow V \text{ such that } \varphi(t^{(\ell)}t) = 
\pi({t^{(\ell)}})\varphi(t)\text{ for all }t^{(\ell)}\in \mathcal{T}^{(\ell)},
t\in \mathcal{T},
\end{equation*}
where $i\in I$ acts by $\varphi \mapsto \varphi(\textvisiblespace\,
\overline{i})$. It is clear that the assignment $\varphi \mapsto 
\varphi\circ \theta$ gives rise to the desired isomorphism.
\end{proof}
\begin{prop}
Let $\Gamma$ be a (locally) profinite group, 
$H$ a normal compact subgroup such that $\Gamma /H$ is a finite
$\ell$-group. If $\chi$ is an $H$-character which admits a continuation
to $\Gamma$, then $\operatorname{ind}_H^{\Gamma}(\chi)$ is
indecomposable. 
\end{prop}
\begin{proof}
Denote the continuation by $\tilde{\chi}$, then we have
$\operatorname{ind}_H^{\Gamma}(\chi) =
\tilde{\chi}\otimes \operatorname{ind}_H^{\Gamma}(1)$
by \cite{VigLivre}, I.5.2.d. 
It is an elementary observation that $\operatorname{ind}_H^{\Gamma}(1)
\cong R[H\backslash\Gamma]$, where $\gamma\in \Gamma$ acts on 
$R[H\backslash\Gamma]$ 
by multiplication
with $\gamma^{-1}$ from the right.
By Brauer theory, 
$R[H\backslash\Gamma]$ is indecomposable.
($H\backslash\Gamma$ is a finite $\ell$-group,
hence its group-algebra decomposes into indecomposable blocks, and there
is a $1$-to-$1$-correspondence between these blocks and $\ell$-regular classes.
There is only the trivial $\ell$-regular class.)
\end{proof}
\begin{lem}
$\operatorname{ind}_{(\mathbb{Z}_p^{\times})^{(\ell)}}^{
\mathbb{Z}_p^{\times}}(\chi_i)$ is the projective cover of
$\chi_i$ and $\operatorname{ind}_{\mathcal{T}^{(\ell)}}^{\mathcal{T}}(
\overline{\chi}_i)$ is the projective cover of $\overline{\chi}_i$.
\end{lem}
\begin{proof}
The induced representation is projective (see also \cite{VigLivre}, I.4.6, in the
$\mathbb{Z}_p^{\times}$-case) 
and admits $\chi_i$ (resp. $\overline{\chi}_i$)
as a quotient. By the above proposition it is also clear that 
it is indecomposable. This is sufficient to conclude the statement
by \cite{VigLivre}, A.4.
\end{proof}
Now the preceding results
yield
\begin{equation*}
\mathfrak{R}^{[(M, \pi)]} \;\cong\;
\mathscr{H}(G, I^{(\ell)}, \chi|I^{(\ell)})\!-\!\textbf{Mod}
\end{equation*}
and
\begin{eqnarray*}
\mathscr{H}(G, I^{(\ell)}, \chi|I^{(\ell)}) \cong
\mathscr{H}(\mathbb{Q}_p, (\mathbb{Z}_p^{\times})^{(\ell)}, \chi_1) 
\otimes
\mathscr{H}(\mathbb{Q}_p, (\mathbb{Z}_p^{\times})^{(\ell)}, \chi_2)\qquad\\
\qquad\qquad\cong 
\bigotimes_{\text{two copies}}
R\left[\mathbb{Q}_p/ (\mathbb{Z}_p^{\times})^{(\ell)}\right],
\end{eqnarray*}
where
\begin{equation*}
R\left[\mathbb{Q}_p/ (\mathbb{Z}_p^{\times})^{(\ell)}\right]\!-\!
\textbf{Mod}
\;{\cong} \;\text{Unipotent block of }
\mathfrak{R}\bigl(\operatorname{GL}_1(F)\bigr).
\end{equation*}

\providecommand{\MR}{\relax\ifhmode\unskip\space\fi MR }
\providecommand{\MRhref}[2]{%
  \href{http://www.ams.org/mathscinet-getitem?mr=#1}{#2}
}
\providecommand{\href}[2]{#2}

\end{document}